\documentclass[12pt]{article}
\usepackage{amsfonts,amsmath,amsthm, latexsym, amssymb, mathrsfs}
\usepackage{bm, wasysym}
\def\ni{\noindent}
\def\ee{\varepsilon}
\def\beq{\arraycolsep1pt\begin{eqnarray*}}
\def\eeq{\end{eqnarray*}}

\newcommand{\MM}{{\cal M}}

\newcommand{\B}{{\mathbb B}}

\newcommand{\Z}{{\mathbb Z}}

\newcommand{\R}{{\mathbb R}}

\newcommand{\cat}{{\rm{cat}}}
\newcommand{\cl}{{\rm{cl}}}
\newcommand{\hess}{{\rm{Hess}}}
\newcommand{\dist}{{\rm{dist}}}
\newcommand{\half}{{\textstyle{\frac{1}{2}}}}
\font\caps=cmcsc10

\begingroup
\newtheorem{theorem}{Theorem}[section]
\newtheorem*{theorem*}{Theorem}
\newtheorem{lemma}[theorem]{Lemma}
\newtheorem{proposition}[theorem]{Proposition}

\newtheorem{definition}[theorem]{Definition}

\endgroup

\title{A Poincar\'e\,--\,Birkhoff theorem for Hamiltonian flows\\ on nonconvex domains.}

\author{Alessandro Fonda and Antonio J. Ure\~na\footnote{Corresponding author}}

\date{}

\begin{document}

\maketitle

\begin{quote}
\small
{\caps Abstract.} 
We present a higher-dimensional version of the Poincar\'e-Birkhoff theorem which applies to Poincar\'e time maps of Hamiltonian systems.  The maps under consideration are neither required to be close to the identity nor to have a monotone twist. The annulus is replaced by the product of an $N$-dimensional torus and the interior of a $(N-1)$-dimensional (not necessarily convex) embedded sphere; on the other hand, the classical boundary twist condition is replaced by an avoiding rays condition.
\end{quote}

\begin{section}{Introduction}

The Poincar\'e\,--\,Birkhoff theorem is a foundational result of symplectic topology, first inferred by Poincar\'e shortly before his death~\cite{Poi}, and proved in its full generality by Birkhoff some years later~\cite{Bir1, Bir3}. In broad terms, it  states the existence of at least two fixed points of an area-preserving homeomorphism of the (closed) planar annulus, provided that it keeps both boundary circles invariant, while rotating them in opposite senses. It has been extended in different directions, and subsequently widely applied to the study of the dynamics of (planar) Hamiltonian systems. See, e.g.,~\cite{BroNeu} or~\cite[Ch. 2]{MosZeh} for a more precise description of this result, and~\cite{FonSabZan, LeC} for two recent review papers, including corresponding lists of references.

\medbreak

The efforts to generalize this theorem to higher dimensions go back to Birkhoff himself~\cite{Bir}, and have been later continued in many works, including~\cite{Arn, BosOrt, ConZeh, FonUre1, Gol, Mos, MosZeh, Szu1, Wei}. Out of these extensions, we shall be particularly interested in the version due to Moser and Zehnder~\cite[Theorem~2.21, p. 135]{MosZeh}, which we recall next. Let $\mathcal S\subseteq\R^N$ be a smooth, compact and convex hypersurface bounding some open region ${\rm int}\,\mathcal S$, and let the smooth map $\mathcal P:\R^N\times\overline{{\rm int}\,\mathcal S}\to\R^N\times\R^N$ 
be given. It is assumed to be an exact symplectic diffeomorphism into its image and to have the form
\begin{equation}\label{P}
\mathcal P(x,y) = (x + \vartheta(x,y),\rho(x,y))\,,\qquad (x,y)\in\R^N\times\overline{{\rm int}\,\mathcal S}\,,
\end{equation}
where both maps $\vartheta,\rho$ are $2\pi$-periodic in each of the first $N$ variables $x_1,\ldots, x_N$.  Suppose further that there exists some $c\in{\rm int}\,\mathcal S$ such that
\begin{equation}\label{twist}
\langle \vartheta(x,y),y-c\rangle>0\,,\;\hbox{ \rm{for every} }(x,y)\in\R^N\times\mathcal S\,.
\end{equation}
Under the additional condition that $\mathcal P$ is either close to the identity or satisfies a monotone twist condition, the Moser\,--\,Zehnder theorem ensures the existence of at  least $N+1$ geometrically distinct  fixed points of $\mathcal P$ in $\R^N\times{{\rm int}\,\mathcal S}$. Incidentally, we observe that this result is also valid if inequality (\ref{twist}) is reversed.

\medbreak
     
A natural question which appears here is {\em how this theorem can be extended to nonconvex hypersurfaces $\mathcal S$}. In this paper we answer to this question in the case that $\mathcal S$ is diffeomorphic to the sphere and $\mathcal P$ is the Poincar\'e map associated to some time-periodic Hamiltonian system. In addition, we shall get rid of the closedness to the identity/monotone twist requirements.

\medbreak

More precisely, we consider the  Hamiltonian system
$$
\dot z=J\nabla H(t,z)\,,
\leqno{\hspace{5mm}(HS)}
$$
where
$J=
{\footnotesize{
\Big(
\begin{array}{cc}
0 & I_N \\
-I_N & \;0
\end{array}
\Big) }} 
$ 
denotes the standard $2N\times 2N$ symplectic matrix. The Hamiltonian
$H:\R\times\R^{2N}\to\R$, $H= H(t,z)=H(t,x,y)$
 is assumed $T$-periodic in the time variable $t$, $2\pi$-periodic in the
first $N$ state variables $x_1,\dots,x_N$, and continuously
differentiable with respect to $z=(x,y)$, with associated gradient $\nabla H=\nabla_z H$. 

\medbreak

Once a $T$-periodic solution $z(t)=(x(t),y(t))$ has been
found, many others appear by just adding an integer multiple of
$2\pi$ to some of the components $x_i(t)$; for this reason, we will
call {\em geometrically distinct} two periodic solutions of $(HS)$  which can not be obtained from each other
in this way. If the Hamiltonian $H=H(t,z)$ is $C^2$-smooth with respect to $z$, the $T$-periodic solution $z$ is called {\em nondegenerate} if $1$ is not a Floquet multiplier of the associated linearized equation.

\medbreak

Let $\mathcal S\subseteq\R^N$ be a ($C^1$-smooth) embedded sphere, i.e., a $C^1$-submanifold which is $C^1$-diffeomorphic to the standard sphere $\mathbb S^{N-1}$, and let  $\nu:\mathcal S\to\R^N$ be the (continuously defined) unit outward normal vectorfield.
At a point $y\in\mathcal S$, the inward and outward rays are defined by
$$
\mathscr R_-(y):=\big\{-\lambda \nu(y):\lambda\ge0\big\}\,,\qquad\qquad \mathscr R_+(y):=\big\{\lambda \nu(y):\lambda\ge0\big\}\,,
$$ 
respectively. We shall say that the flow of the Hamiltonian system~$(HS)$ satisfies the {\em  avoiding outward rays condition} relatively to ${\cal S}$ if every solution $z(t)=(x(t),y(t))$ of~$(HS)$ with $y(0)\in\overline{{\rm int}\,{\cal S}}$ is defined for every $t\in[0,T]$, and if $y(0)\in\mathcal S$ then one has
\begin{equation}
\label{arc}x(T)-x(0)\notin{\mathscr R}_+(y(0))\,.
\end{equation}
The avoiding inward rays condition is defined by replacing ${\mathscr R}_+(y(0))$ with ${\mathscr R}_-(y(0))$ in the definition above. The main result of this paper is the following:

\begin{theorem}\label{thmmain1}
 Let  the flow of~$(HS)$ satisfy the avoiding outward (or inward) rays condition relatively to the embedded sphere ${\cal S}$. Then, $(HS)$ has at least $N+1$ geometrically distinct $T$-periodic solutions $z^{(0)},\dots,z^{(N)}$ with 
	$
	z^{(k)}(0)\in\mathbb R^N\times{\rm int}\,{\mathcal S}\,,\hbox{ \rm for }k=0,\dots,N.$
	In addition, if the Hamiltonian function $H=H(t,z)$ is twice continuously differentiable with respect
	to $z$ and the $T$-periodic solutions with initial condition on $\R^N\times {\rm int}\,{\mathcal S}$ are nondegenerate, then there are at least $2^N$ of them.
\end{theorem}
We observe that, the Hamiltonian function $H=H(t,z)$ being just continuously differentiable with respect to $z$,  uniqueness for initial value problems may fail and the Poincar\'e time-$T$ map $\mathcal P$ may be multivalued. However, in case uniqueness for initial value problems holds, then $\mathcal P$ has the form (\ref{P}) for some maps $\vartheta,\rho$ which are $2\pi$-periodic in each variable $x_i$. If, moreover, the embedded sphere $\cal S$ is convex, then the avoiding outward rays condition (\ref{arc}) is less restrictive than the Moser-Zehnder condition (\ref{twist}).

\medbreak

The avoiding rays conditions have implications on the Brouwer degree of the maps $\vartheta(x,\cdot):\overline{{\rm int}\,\mathcal S}\to\mathbb R^N$. Indeed, under the avoiding {\em inward} rays condition, ${\rm deg}\big(\vartheta(x,\cdot),{\rm int}\,\mathcal S,0\big)=1$ for every $x\in\mathbb R^N$, while if the avoiding {\em outward} rays condition holds, ${\rm deg}\big(\vartheta(x,\cdot),{\rm int}\,\mathcal S,0\big)=(-1)^N$ for every $x\in\mathbb R^N$. This is easy to check, since the maps $\vartheta(x,\cdot):\overline{{\rm int}\,\mathcal S}\to\mathbb R^N$ can be connected by homotopies to other maps which coincide on  $\cal S$ with the outer normal vectorfield $\nu$ (in the case of the avoiding inward rays condition) or the inner normal vectorfield  $-\nu$ (in the case of the avoiding outward rays condition); subsequently, the degrees can be computed, e.g., by using the main theorem of \cite{Ama}. We do not know whether the avoiding rays conditions in Theorem \ref{thmmain1} can be replaced by these more general assumptions on the topological degrees.

\medbreak

In a recent paper \cite{FonUre1} we have shown a result which includes Theorem \ref{thmmain1} when $\mathcal S$ is convex. The strategy of the proof consisted roughly in modifying the Hamiltonian so that it becomes quadratic in a neighborhood of infinity and subsequently applying  some theorems by Szulkin \cite{Szu1,Szu2}. A key step in this procedure consisted in finding a smooth function $h:\mathbb R^N\to\mathbb R$ with 
$$h\equiv 0\text{ on }{\rm int}\,\mathcal S\,,\qquad \nabla h(y)\not =0\text{ for all }y\in{\rm ext}\,\mathcal S\,,\qquad h(y)=|y|^2\text{ for $|y|$ large enough.}$$ We do not know whether such a function exists in general if $\mathcal S$ is a nonconvex embedded sphere. Instead, in Section \ref{sec3} we construct a function $h$ (which we call a {\em basket function}) which satisfies the first two properties above and  has a finite limit at infinity. This leaves us outside the framework of \cite{Szu1,Szu2}, prompting us to develop a new general result (Theorem \ref{szumod}) on the existence and multiplicity of periodic solutions for some Hamiltonian systems $(HS)$ which are periodic in the $x_i$ variables and have a finite limit when $|y|\to+\infty$. Theorem~\ref{szumod}, in its turn, will be deduced from an abstract critical point theorem for strongly indefinite functionals (Theorem~\ref{absth1}). Finally, the proof of Theorem~\ref{absth1} will be divided between Section~\ref{sec6} (for the general, possibly degenerate case) and Sections~\ref{sec7}-\ref{sec8} (for the nondegenerate case). We collect some known material needed in Sections \ref{sec6}-\ref{sec7} in the Appendix (Section \ref{appendix}), at the end of the paper.

\medbreak

If the number of degrees of freedom is $N=1$ then the embedded sphere $\mathcal S$ becomes a two-point set and Theorem \ref{thmmain1} becomes the particularization  for Hamiltonian systems of the usual Poincar\'e-Birkhoff theorem, see e.g. \cite{BroNeu} (or  \cite[Theorem 1.1]{FonUre1} if there is no uniqueness). For these reasons, throughout this paper  we shall be concerned only with dimensions $N\geq 2$.

\end{section}

\begin{section}{Embedded spheres in $\R^N$}\label{sec3}

 The goal of this section is to present a couple of elementary features of embedded spheres in $\R^N$: a regularization result and the existence of some associated functions which we  call `basket functions'. We begin with the regularization result; with this aim we choose some $C^1$-smooth embedded sphere  $\mathcal S\subseteq\R^N$ and denote by  $\nu:\mathcal S\to\mathbb R^N$ to its outward normal vectorfield.

\begin{lemma}\label{intsph}
Let  $K\subseteq{\rm int}\,\mathcal S$ be a compact set and let $\varepsilon>0$ be given. Then,
there exists a $C^\infty$-smooth embedded sphere ${\mathcal S}_*\subseteq\R^N$, with associated outer normal vectorfield $\nu_*:\mathcal S_*\to\mathbb R^N$, satisfying 
\begin{enumerate}
	\item[$(i)$]$K\subseteq{\rm int}\, {\mathcal S_*}\subseteq\overline{{\rm int}\, {\mathcal S_*}}\subseteq{\rm int}\, {\mathcal S}\,.$
		\item[$(ii)$] For any $q\in\mathcal S_*$ there is some $p\in\mathcal S$ with $
		|q-p|<\ee$ and  $|\nu_*(q)-\nu(p)|<\ee\,.$
\end{enumerate}

\end{lemma}

\begin{proof} Using a regularization argument we can find a $C^1$-smooth vector field $X:\mathcal S\to\R^N$ with $|X(p)|=1$ and 
$
\langle X(p),\nu(p)\rangle>0$ for all $p\in\mathcal S$.
Choose some $C^1$-smooth diffeomorphism $\sigma:\mathbb S^{N-1}\to\mathcal S$ and consider the map $\varphi:\mathbb S^{N-1}\times\mathbb R\to\R^N$ defined by
$$
\varphi(\theta,t):=\sigma(\theta)+tX(\sigma(\theta))\,.
$$
If $t>0$ is small enough then $\mathcal S_*(t):=\varphi(\mathbb S^{N-1}\times\{t\})$ is a $C^1$-smooth embedded sphere satisfying  {\em (i)-(ii)}, and $\varphi(\cdot,t):\mathbb S^{N-1}\to\mathcal S_*(t)$ is a $C^1$-diffeomorphism. The result follows by setting $\mathcal S_*:=\phi(\mathbb S^{N-1})$ for some  $C^\infty$ map $\phi:\mathbb S^{N-1}\to\mathbb R^N$  sufficiently close, in the $C^1$ sense, to $\varphi(\cdot,t)$. 
\end{proof}

In~\cite{Mor}, Morse proved a differentiable version of the Schoenflies Theorem which we recall below for the reader's convenience. We denote by $\mathbb B^N$ and $\overline{\mathbb B^N}$ to the open unit ball  in $\mathbb R^N$ and its closure, respectively.  
\begin{theorem}[Morse, \cite{Mor}]\label{mor}{Given a $C^2$-smooth embedded sphere  $\widehat{\cal S}\subseteq\R^N$ there are open sets $\widehat{\mathscr U}\supset\overline{\mathbb B^N}$ and $\widehat{\mathscr V}\supset{\overline{{\rm int}\,{\mathcal S}}}$, a point $p\in\mathbb B^N$, and a homeomorphism $\widehat{\mathcal F}:\widehat{\mathscr U}\to\widehat{\mathscr V}$, such that $$\widehat{\mathcal F}(\mathbb S^{N-1})=\mathcal S,\qquad\widehat{\mathcal F}(\mathbb B^N)={\rm int}\,\mathcal S\,,\qquad\widehat{\mathcal F}:\widehat{\mathscr U}\backslash\{p\}\to\widehat{\mathscr V}\backslash\{\widehat{\mathcal F}(p)\} \text{ is a $C^2$ diffeomorphism.}$$}
	\end{theorem}  
We shall need a variant of this result in which the exterior of the standard sphere is mapped into the exterior of an embedded sphere. This is the content of Lemma \ref{diffeo} below.
\begin{lemma}\label{diffeo}
	{Given a $C^2$-smooth embedded sphere ${\cal S}\subseteq\R^N$, there are open sets $\mathscr U\supset\mathbb{R}^N\setminus\mathbb B^N$ and $\mathscr V\supset{\overline{{\rm ext}\,{\mathcal S}}}$, and a $C^2$-smooth diffeomorphism $\mathcal F:\mathscr U\to\mathscr V$ such that $\mathcal F(\mathbb S^{N-1})=\mathcal S$ and $\mathcal F(\R^N\setminus\overline{\mathbb B^N})={\rm ext}\,\mathcal S$. }
\begin{proof} There is no loss of generality in assuming that $0\in{\rm int}\,{\mathcal S}$. We denote by $\mathscr I:\mathbb R^N\backslash\{0\}\to\mathbb R^N\backslash\{0\},\ y\mapsto y/|y|^2$ to the inversion map and fix $\widehat{\mathscr U},\ \widehat{\mathscr V},\ p$ and $\widehat{\cal F}$ as given by Theorem \ref{mor} for the embedded sphere $\widehat{\mathcal  S}:=\mathscr I(\mathcal S)$. An elementary argument\footnote{We are indebted to R. Ortega for pointing this reference to us.}~\cite[\S(16.26.8)]{Die} shows that for any any open connected set $G\subseteq\mathbb R^N$ and any two points $p,q\in G$ there exists a $C^\infty$-smooth diffeomorphism of $\mathbb R^N$ sending $p$ into $q$ while leaving every point in $\mathbb R^N\backslash G$ fixed; applying this result to $G=\mathbb B^N$ and $G={\rm int}\,\widehat{\mathcal S}$ we see that we can set $p=0=\widehat{\mathcal F}(p)$.  To conclude the construction we define $\mathscr U:=\mathscr I(\widehat{\mathscr U}\backslash\{0\})$, $\mathscr V:=\mathscr I(\widehat{\mathscr V}\backslash\{0\})$, and $\mathcal F:=\mathscr I\circ\widehat{\mathcal F}\circ\mathscr I:\mathscr U\to\mathscr V$. The proof is complete.
\end{proof}

\end{lemma}

\begin{definition}[Basket functions for embedded spheres]\label{bf} Let $\mathcal S\subseteq \mathbb R^N$ be an embedded sphere (or any compact hypersurface); by a {\em basket functions} for $\mathcal S$ we mean a $C^2$-smooth function $h:\R^N\to[0,1[$ satisfying:
\begin{enumerate}
	\item[(a)]  $h(y)=0$, {\rm for every} $y\in{{\rm int}}\,{\mathcal S}$\,;
	\item[(b)] $\nabla h(y)\ne0$, {\rm for every} $y\in{\rm ext}\,{\mathcal S}$\,;
	\item[(c)] $\displaystyle\lim_{\underset{y\in{\rm ext}\,{\mathcal S}} {y\to y_0}}\frac{\nabla h(y)}{|\nabla h(y)|}=\nu(y_0)$, {\rm for every} $y_0\in{\cal S}$\,;
	\item[(d)] $\displaystyle\lim_{|y|\to\infty} h(y)=1\,;\qquad\lim_{|y|\to\infty} \nabla h(y)=0\,;\qquad\lim_{|y|\to\infty} \hess\, h(y)=0\,.$
\end{enumerate}  
\end{definition}

\begin{lemma}\label{constr}
Every embedded sphere $\mathcal S\subseteq\R^N$ of class $C^2$ admits a basket function.
\end{lemma}

\begin{proof}Choose the open neighborhoods of infinity $\mathscr U,\mathscr V\subseteq\R^N$ and the $C^2$-smooth diffeomorphism $\mathcal F:\mathscr U\to\mathscr V$ as given by Lemma~\ref{diffeo}. We consider the $C^2$ function $f:\mathscr V\to\R$ defined by
	$$
	f(y):=|{\cal F}^{-1}(y)|^2-1\,,\qquad y\in\mathscr V\,.
	$$
	Notice that $f^{-1}(\{0\})=\mathcal S$ and $f^{-1}(]0,+\infty[)={\rm ext}\,\mathcal S$. Moreover,
	\begin{equation*}
	\nabla f(y)\not=0\ \forall y\in\overline{{\rm ext}\,\mathcal S}\,,\qquad \frac{\nabla f(y)}{|\nabla f(y)|}=\nu(y)\ \forall y\in\mathcal S\,,\qquad\lim_{|y|\to\infty}f(y)=+\infty\,.
	\end{equation*}
Observe also that the level sets $\mathcal S_r:=f^{-1}(r)$ are embedded spheres for all $r\geq 0$. We define 
	$$m:[0,+\infty[\,\to]0,+\infty[\,,\qquad
	m(r):=(r+1)\max\{|\nabla f(y)|+\|\hess\, f(y)\|:y\in\mathcal S_r\}\,.
	$$
Since $m$ is continuous and positive, it is possible to find a $C^2$-smooth function $g:\R\to\R$ with $g(r)=0\text{ for }r\leq 0$, $g'(r)>0\text{ for }r>0$, $\lim_{r\to+\infty}g(r)=1$ and 
	 \begin{equation}\label{eu3}
 -\frac{1}{m(r)^2}<g''(r)<0<g'(r)<\frac{1}{m(r)}\,,\qquad\text{if }r\geq 1\,.
	 \end{equation}
	Define $h:\R^N\to\R$ by
	$$
	h(y):=
	\begin{cases}
	\vspace{2mm}
	0\,, & \hbox{ if }y\in\overline{{\rm int}\,{\mathcal S}}\\
	g(f(y))\,, & \hbox { if }y\in{\rm ext}\,{\mathcal S}
	\end{cases}\,.
	$$
	Since $g(f(y))=0$ on $(\overline{{\rm int}\,{\mathcal S}})\cap\mathscr V$, the function $h$ is well defined and $C^2$-smooth. Properties {\em (a),(b),(c)}, as well as the first statement in {\em (d)} are now immediate. Concerning the second part of  {\em (d)} we combine the first two inequalities of (\ref{eu3}) and the definition of $m$ to get the estimates 
	$$y\in{\rm ext}\,\mathcal S_1\Rightarrow|\nabla h(y)|<\frac{1}{m(f(y))}|\nabla f(y)|\leq\frac{1}{1+f(y)}\to 0\text{ as }|y|\to\infty\,.$$
	
	Finally in order to check the last assertion of  {\em (d)} we observe that, by the triangle inequality,
	$$y\in{\rm ext}\,{\mathcal S_1}\Rightarrow\|{\rm Hess}\, h(y)\|\leq|g''(f(y))|\,|\nabla f(y)|^2+|g'(f(y))|\,\|{\rm Hess}\,f(y)\|\,,$$
which, in combination with (\ref{eu3}) and the definition of $m$, gives
	$$y\in{\rm ext}\,{\mathcal S_1}\Rightarrow\|{\rm Hess}\, h(y)\|<\frac{|\nabla f(y)|^2}{m(f(y))^2}+\frac{\|{\rm Hess}\, f(y)\|}{m(f(y))}\leq\frac{1}{(1+f(y))^2}+\frac{1}{1+f(y)}\to 0\text{ as }|y|\to\infty\,.$$

The proof is complete.
\end{proof}

\end{section}

\begin{section}{Modifying the Hamiltonian}\label{sec4}

This section, which goes along the lines of Sections 4 and 5 of \cite{FonUre1}, is devoted to prove Theorem \ref{thmmain1}. An important ingredient of the proof is Theorem \ref{szumod} below. This result is reminiscent of some theorems by Szulkin (cf.~\cite[Theorem 4.2]{Szu1} and~\cite[Theorem 8.1]{Szu2}) because, as him, we study existence and multiplicity of periodic solutions for some Hamiltonian systems $H=H(t,x,y)$ which are periodic in the $x_i$ variables. On the other hand, Szulkin studied the case of $H$ being quadratic in the $y_i$ directions while we are interested in $H$ having a finite limit as $|y|\to\infty$. 
\medbreak

We assumed in the Introduction that $H=H(t,x,y)$ is periodic in the time variable $t$. It will be convenient to relax this assumption now and consider instead Hamiltonians  which are defined only for $t\in[0,T]$ and are not necessarily the restriction of a continuous, time-periodic function. 
\begin{definition}[Admissible Hamiltonians]\label{dad} The Hamiltonian function $H:[0,T]\times\mathbb R^{2N}\to\mathbb R$ will be called {\em admissible} if it is continuous and $2\pi$-periodic in the $x_i$ variables, and its first-order partial derivatives with respect to $z=(x,y)$ are continuously defined on $[0,T]\times\mathbb R^{2N}$. 
	\end{definition}
	By some abuse of notation, we shall also say that the solution $z:[0,T]\to\mathbb R^{2N}$ of $(HS)$ is $T$-periodic provided that $z(0)=z(T)$.

\begin{theorem}\label{szumod}
	Let the Hamiltonian function $H:[0,T]\times\R^{2N}\to\R$ be admissible. Assume that
	
	\begin{enumerate}
		\item[{\bf [}${\bm H_1}${\bf ]}] 
		there exists  $R_0>0$ such that $H(t,x,y)\equiv h(y)$ does not depend on $t,x$ if $|y|\ge R_0$,
		\item[{\bf [}${\bm H_2}${\bf ]}] $h$ has a finite limit $\ell$ as $|y|\to\infty$; furthermore, $h(y)\not=\ell$ for $|y|\geq R_0$,
		\item[{\bf [}${\bm H_3}${\bf ]}] $\lim_{|y|\to\infty}\nabla h(y)=0$.
	\end{enumerate}
	Then, system~$(HS)$ has at least $N+1$ geometrically distinct $T$-periodic solutions. If, in addition: 
	\begin{enumerate}
		\item[{\bf [}${\bm H_4}${\bf ]}] $H$ is $C^2$-smooth with respect to $(x,y)$,
		\item[{\bf [}${\bm H_5}${\bf ]}] $\lim_{|y|\to\infty}\hess\, h(y)=0$, 
		\item[{\bf [}${\bm H_6}${\bf ]}] the $T$-periodic solutions of $(HS)$ are nondegenerate, 
	\end{enumerate}
	then $(HS)$ has at least~$2^N$ geometrically different $T$-periodic solutions.  
\end{theorem}

In this result, assumption {\bf [}${\bm H_1}${\bf ]} looks too strong for applications, and it seems plausible that it could be avoided if one replaces the limits in {\bf [}${\bm H_{2,3,5}}${\bf ]} by their analogues referred to $H(t,x,y)$, assumed uniform with respect to $(t,x)$. However, this is what we shall need for the purposes of this paper.

\medbreak

 The proof of Theorem \ref{szumod} is postponed to next section. The remaining of this one is devoted to show how it implies Theorem \ref{thmmain1}; with this aim, we need a reinforced version of Definition \ref{dad}

\begin{definition}[Strongly admissible Hamiltonians]\label{sah}
	 The Hamiltonian function $H:[0,T]\times\R^{2N}\to\R$ is said to be strongly admissible (with respect to the embedded sphere ${\mathcal S}\subseteq\mathbb R^N$) provided that it is admissible and the following two conditions hold: 
	
	\begin{enumerate}
		\item[{\rm\bf[1.]}] there exists a relatively open set ${\cal V}\subseteq[0,T]\times\R^N$, containing $\{0\}\times\overline{{\rm ext}\,{\cal S}}$, such that $H$ is $C^\infty$-smooth with respect to $z=(x,y)$ on the set ${\cal V}_\sharp:=\{(t,x,y): (t,y)\in{\cal V}, \,x\in\R^N\}$; 
		\item[{\rm\bf [2.]}] there exists some $R>\max_{y\in\mathcal S}|y|$ such that
		$H(t,x,y)=0$, if $|y|\ge R$.
	\end{enumerate}
\end{definition}

We shall start by observing that it suffices to prove Theorem~\ref{thmmain1} under the additional assumptions that the embedded sphere $\mathcal S$ is $C^2$-smooth and the Hamiltonian $H$ is strongly admissible. With this aim, let the admissible Hamiltonian $H:[0,T]\times\mathbb R^{2N}\to\mathbb R$ and the $C^1$-smooth embedded sphere $\mathcal S\subseteq\R^N$ be given, and assume that the flow of $H$ satisfies the avoiding outward rays condition (\ref{arc}) (resp., the avoiding inward rays condition) relative to $\mathcal S$. We denote by $(\widehat{HS})$ the Hamiltonian system associated with the modified Hamiltonian $\widehat H$.

\begin{lemma}\label{lem272}
	{Under the above, there exists a strongly admissible Hamiltonian $\widehat H$ and a $C^2$-smooth embedded sphere $\widehat{\cal S}\subseteq{\rm int }\,\mathcal S$, such that:
		\begin{enumerate}
			\item[{$(\varhexstar)$}]  $H$ and $\widehat H$ coincide on some relatively open set containing the graph of every $T$-periodic solution $(\widehat x,\widehat y)$ of  $(\widehat{HS})$ starting with $\widehat y(0)\in{\rm int}\,\widehat{\cal S}$;
			\item[{$(\varhexstar\varhexstar)$}] the flow of $\widehat H$ satisfies the avoiding outward rays condition  (resp., the avoiding inward rays condition) relative to $\widehat{\cal S}$.
	\end{enumerate}}
\end{lemma}

The proof of Lemma \ref{lem272} follows from the combination of Lemma \ref{intsph}, a regularization argument and the Ascoli-Arzel\`a Theorem (see \cite[Section 4]{FonUre1}), and will be omitted. What remains of the proof of Theorem \ref{thmmain1} is also very similar to \cite[Section 5]{FonUre1} with only some small changes (for instance, the set $D$ appearing there has to be replaced by ${\rm int}\,\mathcal S$). We sketch it below, for the reader's convenience.

\begin{proof}[Proof of Theorem \ref{thmmain1}] We assume, to fix ideas, that the flow of $(HS)$ satisfies the avoiding {\em outward} rays condition (\ref{arc}) with respect to $\mathcal S$. 
	In view of Lemma \ref{lem272} there is no loss of generality in further assuming that the embedded sphere $\mathcal S$ is $C^2$-smooth and the Hamiltonian is strongly admissible; thus, we let $\mathcal V,\mathcal V_\sharp$ and $R>0$ be as given by Definition \ref{sah}.   Let the basket function $h:\mathbb R^N\to[0,1[$ be given by Lemma \ref{constr} and choose positive numbers $0<\tau<T$ and $\varrho>0$  small enough so that 
	\begin{equation}
	\label{eu0}	[0,\tau]\times \{\eta\in\overline{\mathbb B}_R:\dist(\eta,{\rm ext}\,\mathcal S)\leq\varrho\}\subseteq\mathcal V\,.
	\end{equation}
By combining the boundedness of $\nabla H$, the avoiding rays condition (\ref{arc}) and item {\em (c)} in Definition \ref{bf} we see that if $\tau$ and $\varrho$ are chosen small enough, then  the following implications hold  for solutions $z(t)=(x(t),y(t))$ of~$(H\!S)$:
	\begin{align}
	&t\in[0,\tau]&\Rightarrow & &|y(t)-y(0)|<\varrho\,,\label{eu4}\\
&\dist(y(0),\mathcal S)\le\varrho&\Rightarrow & &x(T)-x(0)\not\in\{\lambda\nabla h(y(0)):\lambda\geq 0\}\,.\label{eu5}
\end{align}
(We denote by $\overline{\mathbb B}_R$ to the closed ball of radius $R$ in $\mathbb R^N$).	We consider the sets 
	$$
 \Delta:=\Big([0,\tau]\times \overline{{\rm ext}\,\mathcal S} \Big)\cup\Big([0,T]\times(\mathbb R^N\backslash\overline{\mathbb B}_R)\Big)\,,\qquad \Delta_\sharp:=\{(t,\xi,\eta)\in[0,T]\times\mathbb R^{2N}:(t,\eta)\in\Delta\}\,.
	$$  

Following the construction in \cite[Lemma 5.2]{FonUre1} one can find a $C^2$-smooth function $r:[0,T]\times\R^{N}\to\mathbb R$ and some $R_0>R$ satisfying
\begin{enumerate}
	\item[$(\star)$] $r(t,\eta)=0$, if $(t,\eta)\notin\Delta$,
	\item[$(\star\star)$] $\frac{1}{T}\int_0^Tr(t,\eta)\,dt=h(\eta)$, for every $\eta\in\R^N$,
	\item[$(\star$$\star\star)$] $r(t,\eta)=h(\eta)$, if $|\eta|\geq R_0$.
\end{enumerate}

Let $(\mathcal X(t;\xi,\eta),\mathcal Y(t;\xi,\eta))$ be the value at time $t$ of the solution of  $(HS)$ starting from $(\xi,\eta)\in\mathbb R^{2N}$ at time $t=0$. In view of (\ref{eu0})-(\ref{eu4}) and requirements {\rm\bf [1.]-[2.]} in Definition \ref{sah}, the map $\mathfrak Z(t;\xi,\eta):=(t;\mathcal X(t;\xi,\eta),\mathcal Y(t;\xi,\eta))$ defines a $C^\infty$ diffeomorphism between $\Delta_\sharp$ and its image, which is contained into $\mathcal V_\sharp\cup\big([0,T]\times\mathbb R^N\times(\mathbb R^N\backslash\overline{\mathbb B}_R)\big)$. This allows us to consider  the functions $r_\sharp,\mathcal R:[0,T]\times\R^{2N}\to\R$ defined by
$$
r_\sharp(t,\xi,\eta)=r(t,\eta)\,,\qquad {\cal R}(t,z)=
\begin{cases}
r_\sharp(\mathfrak Z^{-1}(t,z))\,, &\hbox{ if }(t,z)\in\mathfrak Z({\Delta_\sharp})\,,\\
0\,, & \hbox{ otherwise}\,.
\end{cases}
$$ 

Both of them are $C^2$-smooth. Let the modified Hamiltonian $\widetilde H:[0,T]\times\mathbb R^{2N}\to\mathbb R$ be defined by
$\widetilde H(t,z):=H(t,z)-\lambda{\cal R}(t,z)\,,$
where $\lambda>0$ is some constant. We observe that, for $|y|>R_0$ ,
$$
\widetilde H(t,x,y)=-\lambda{\cal R}(t,x,y)=-\lambda r(t,y)=-\lambda h(y)\,,
$$
and one easily checks that for $\lambda>0$ large enough, $\widetilde H$ lies under the framework of Theorem~\ref{szumod}. Applying this result we obtain at least $N+1$ geometrically distinct $T$-periodic solutions of the associated Hamiltonian system~$(\widetilde{H\!S})$, or $2^N$ geometrically distinct $T$-periodic solutions if the Hamiltonian is twice continuously differentiable with respect to $(x,y)$ and the solutions are nondegenerate. 

\medbreak

To end the argument one can use (\ref{eu5}) to show that, if $\lambda>0$ is large enough system $(\widetilde{H\!S})$ does not have $T$-periodic solutions $z=(x,y)$ starting with $y(0)\in\overline{{\rm ext}\,\mathcal S}$. On the other hand, the Hamiltonian systems $(HS)$ and $(\widetilde{H\!S})$ have the same solutions starting with $y(0)\in{\rm int}\,\mathcal S$, see \cite[Section 5]{FonUre1} for the details. The proof is complete.
\end{proof}

\end{section}

\begin{section}{Critical point theory}\label{sec5}
 Let ${\cal V}$ be a compact manifold without boundary of class $C^2$,  and let $f:\mathcal V\to\R$ be a function of class $C^1$. The maximum and the minimum of $f$ on $\mathcal V$ are critical values of $f$; in particular, $f$ has at least two different critical points.  However, if one assumes some additional complexity in the topology of $\mathcal V$, in some cases it is possible to combine algebraic topology with critical point theory methods to predict the existence of more critical points. For instance, it is well known that if $\mathcal V=(\mathbb R/2\pi\mathbb Z)^N$ is the $N$-torus then every $C^1$ function $f:\mathcal V\to\mathbb R$ has at least $N+1$ different critical points, and $2^N$ if $f$ is $C^2$ and the critical points are nondegenerate.
 
 \medbreak
 
  More generally, to any compact manifold $\mathcal V$ one can associate the integers
  ${\rm cl}(\mathcal V)$ and $ {\rm sb}(\mathcal V)$,
  called, respectively, the {\em cuplength} and the {\em sum of the Betti numbers} of $\mathcal V$. The first one, ${\rm cl}(\mathcal V)$,  is the largest integer $k$ for which there are elements $\alpha_j\in H^{q_j}(\mathcal V),\ j = 1,\ldots, k$ (the singular cohomology vector spaces with real coefficients),  such that $q_j\ge 1$ and the cup product $\alpha_1$\,{\tiny{$\cup$}}\,\dots\,{\tiny{$\cup$}}\,$\alpha_k$ does not vanish.  The second one, ${\rm sb}(\mathcal V)$, is the sum of the dimensions of the singular homology vector spaces with real coefficients $H_n(\mathcal V)$.  Any $C^1$-smooth function on  $\mathcal V$ has at least $\cl(\mathcal V)+1$ critical points, and any $C^2$-smooth function on $\mathcal V$ has at least ${\rm sb}(\mathcal V)$ critical points provided that they are all nondegenerate.  These statements are respective consequences of the Ljusternik-Schnirelmann theorem and the Morse inequality; see, e.g.~\cite[Section 5.2.2]{Cha1},~\cite[Chapter V]{Sch} or~\cite[Section 3.4]{BanHur} for more details. In the case of the $N$-torus $\mathcal V=(\mathbb R/2\pi\mathbb Z)^N$ one has $\cl(\mathcal V)=N$ and ${\rm sb}(\mathcal V)=2^N$.
 
 \medbreak
 
We shall develop results of the kind described above for a class of functionals defined on the product $\MM=E\times\mathcal V$, the Hilbert space $E$ being possibly infinite-dimensional. Our functionals will display a `saddle-like' geometry in the first variable, in line with other results which were amply studied in the literature some 30 years ago, see e.g.~\cite{Cha0, FouWil, Liu, Szu1, Szu2}. However, these references treat cases in which the (global) Palais\,--\,Smale condition holds, and we are here interested in allowing the existence of degenerate directions  along which compactness may fail.   

\medbreak

Precisely, let $E$ be a real Hilbert space, endowed with the scalar product $\langle\cdot,\cdot\rangle$ and the associated norm $\|\cdot\|$. Let $L:E\to E$ be a bounded selfadjoint linear operator, and assume that $E$ splits as the orthogonal direct sum
\begin{equation}\label{splitting}
E=E_-\oplus E_0\oplus E_+\,,
\end{equation}
where $E_0=\ker L\ne\{0\}$ is finite-dimensional and $E_\pm$ are closed subspaces which are invariant for $L$. We further assume that $L$ is positive definite on $E_+$ and negative definite on $E_-$, i.e., there is some constant $\ee_0>0$ such that
\begin{equation}\label{aboutL}
\begin{cases}
\vspace{2mm}
\langle Le_-,e_-\rangle\le-\ee_0\|e_-\|^2,&\hbox{for every }e_-\in E_-\,,\\
\langle Le_+,e_+\rangle\ge\ee_0\|e_+\|^2,&\hbox{for every }e_+\in E_+\,.
\end{cases}
\end{equation}

Let ${\cal V}$ be a finite-dimensional compact $C^2$-smooth manifold without boundary, let $\MM:=E\times\mathcal V$, and let $\psi\in C^1(\MM)$ be given. We shall be interested in the associated  functional $\Phi_\psi:\MM\to\R$, defined by
\begin{equation}\label{functional}
\Phi_\psi(e,v)=\half\langle Le,e\rangle+\psi(e,v)\,.
\end{equation}

\begin{definition}[Class $\mathscr A$] We shall say that $\psi$ belongs to the class $\mathscr A$ provided that:
\begin{enumerate}
\item[{\bf [}${\bm \psi_1}${\bf ]}]\ \  $\psi$ is bounded, there exists some $\ell\in\R$ such that 
$
\psi(e,v)\not=\ell$ for every $(e,v)\in\MM
$, 
and
\begin{equation}\label{lim1}
\mathop{\lim_{\|e_0\|\to \infty}}_{e_0\in E_0}\psi(e_0+b,v)=\ell\,,
\end{equation}
uniformly with respect to $b$ belonging to bounded subsets of $E$ and $v\in\mathcal V$;
\item[{\bf [}${\bm \psi_2}${\bf ]}\ \ ] The partial gradient map $\nabla_E\psi$ is is globally compact (i.e., its image is relatively compact in $E$),  and 
\begin{equation}\label{lim2}
\mathop{\lim_{\|e_0\|\to \infty}}_{e_0\in E_0}\nabla_E\psi(e_0+b,v)=0\,,
\end{equation}
 uniformly with respect to $b$ belonging to bounded subsets of $E$ and $v\in\mathcal V$.  
\end{enumerate}
\end{definition}

\begin{definition}[Class $\mathscr A^+$]
We shall say that the $C^2$-smooth functional $\psi$ belongs to the class $\mathscr A^+$ provided that it  not only belongs to the class $\mathscr A$, but also satisfies:  
\medbreak
\begin{enumerate}
\item[{\bf [}${\bm \psi_3}${\bf ]}]\ \ There exists  $R>0$ such that
$\nabla_{\MM}\psi(e_0,v)\in E_0$  whenever $(e_0,v)\in E_0\times\mathcal V,\ \|e_0\|\ge R$;
\item[{\bf [}${\bm \psi_4}${\bf ]}]\ \ The partial Hessian map $\hess_E\psi:\MM\to {\mathscr L}(E)$ is bounded and satisfies 
\begin{equation}\label{lim3}
\mathop{\lim_{\|e_0\|\to \infty}}_{e_0\in E_0}\hess_E\psi(e_0+b,v)=0\,,
\end{equation}
 uniformly with respect to $b$ belonging to bounded subsets of $E$ and $v\in\mathcal V$;
\item[{\bf [}${\bm \psi_5}${\bf ]}]\ \ all critical points of $\Phi_\psi$ are nondegenerate. 
\end{enumerate}
\end{definition}
\medbreak

In view of (\ref{aboutL}) and {\bf [}${\bm \psi_1}${\bf ]}, the functional $\Phi_\psi$ has a (possibly strongly indefinite) global saddle geometry. The following observation will play an important role in our discussion.
\begin{lemma}\label{lemmaPS}{ If $\psi\in\mathscr A$, then $\Phi_\psi$ satisfies the Palais\,--\,Smale condition (PS)$_c$ for every level $c\ne \ell$. 
	}
\end{lemma}
The proof of this statement, which depends on the fact that $\nabla_E\psi$ is globally compact, follows standard arguments and will be omitted.
We shall show the following:

\begin{theorem}\label{absth1}
{{\bf\em (a).} Assume that $\psi$ belongs to the class $\mathscr A$; then, $\Phi_\psi$ has at least ${\rm cl}(\mathcal V)+1$ critical points. {\bf\em (b).} If $\psi$ belongs to the class $\mathscr A^+$\!\!\!, then $\Phi_\psi$ has at least ${\rm sb}({\cal V})$ critical points.}
\end{theorem}

The two parts of this theorem will be proved, respectively, in Sections~\ref{sec6} and~\ref{sec7}. In this context, there are two additional assumptions which can be introduced without loss of generality and will simplify the proof of Theorem \ref{absth1}. From now on we denote by $\pi_-,\, \pi_0$ and $\pi_+$ to the orthogonal projections of $E$ into $E_-,\, E_0$ and $E_+$, respectively.
\begin{itemize}
\item Firstly, we notice that by an observation made in~\cite[Remark~1.10]{BenCapFor} (see also~\cite[p.~732]{Szu1}), we may assume that
\begin{equation}\label{L}
L=\pi_+-\pi_-\,.
\end{equation} 
Indeed, $\prec e,e'\succ:=\langle (L\pi_+-L\pi_-+\pi_0)e,e'\rangle$ defines an equivalent scalar product on $E$ and $\langle Le,e\rangle=\prec(\pi_+-\pi_-)e, e\succ\ \forall e\in E$.  
\item Secondly, we observe that after possibly changing the signs of  $\psi, L$ and $\ell$, there is no loss of generality in further assuming that 
\begin{equation}
\label{psi}\psi(e,v)<\ell\,,\qquad\text{ for every }(e,v)\in E\times\mathcal V\,. 
\end{equation}
\end{itemize}

A couple of remarks will help prepare the proof of Theorem~\ref{szumod}. To begin with, we observe that a rescaling argument allows us to assume $T=1$. Secondly, without loss of generality we may strengthen  assumption {\bf [}${\bm H_2}${\bf ]} as follows:
\begin{enumerate}
	\item[{\bf [}${\bm H_2^*}${\bf ]}]  $h$ has a finite limit $\ell$ as $|y|\to\infty$; furthermore, $H(t,x,y)\not=\ell$ for all $(t,x,y)\in[0,T]\times\mathbb R^{2N}$.
\end{enumerate}
Indeed, under assumptions {\bf [}${\bm H_{1-3}}${\bf ]} and assuming, to fix ideas, that $h(y)<\ell$ for $|y|\geq R_0$, we may replace $H$ by the Hamiltonian
$$
\breve H(t,x,y):=\begin{cases}H(t,x,y)\,,&\text{if }|y|\leq R_0\,,\\ \vspace{-0.3cm}\\ 
\alpha(h(y))\,,&\text{if }|y|\ge R_0\,,
\end{cases}
$$
where $\alpha:\mathbb R\to\mathbb R$ is a $C^2$ function with $$\alpha(\rho)=\rho\text{ if }\rho\leq\max_{|y|=R_0}h\,,\qquad \alpha'(\rho)>0\text{ for all }\rho\in\mathbb R\,,\qquad\alpha(\ell)>\sup_{[0,T]\times\mathbb R^{2N}}H.$$ This new Hamiltonian inherites {\bf [}${\bm H_{1-3}}${\bf ]} (or  {\bf [}${\bm H_{1-6}}${\bf ]}) from $H$, and further satisfies the reinforced assumption {\bf [}${\bm H_2^*}${\bf ]}. Finally, one easily checks that the corresponding Hamiltonian system has the same periodic solutions.

\medbreak

Let us conclude this section by showing how Theorem~\ref{szumod} follows from the abstract Theorem~\ref{absth1}. To do so we need to write the $1$-periodic solutions of our spatially-periodic Hamiltonian system $(HS)$ as the critical points of a suitable functional defined on the cartesian product of a Hilbert space $E$ and the $N$-torus $(\mathbb R/2\pi\mathbb Z)^N=\mathcal V$. The arguments are mostly well-known (see, e.g.~\cite[Chapter 6]{Rab} or~\cite[Section 4]{Szu1}); for this reason we describe them only briefly.

\begin{proof}[Proof of Theorem \ref{szumod}]Let us denote by $\widetilde H^{\frac{1}{2}}(\R/\Z,\R^N)$ to the closed subspace of $H^{\frac{1}{2}}(\R/\Z,\R^N)$-functions with zero mean. We consider the Hilbert space $E=\widetilde H^{\frac{1}{2}}(\R/\Z,\R^N)\times H^{\frac{1}{2}}(\R/\Z,\R^N)$ and set $\mathcal M:=E\times(\mathbb R/2\pi\mathbb Z)^N$. One identifies each element $z=(x,y)\in H^{\frac{1}{2}}(\R/\Z,\R^N)\times H^{\frac{1}{2}}(\R/\Z,\R^N)$ with $((\tilde x,y),\bar x)\in\mathcal M$, where 
	$$\bar x:=\int_0^1 x(t)dt\ {\rm mod}\ 2\pi\mathbb Z^N\in(\mathbb R/2\pi\mathbb Z)^N\quad\text{ and }\quad\tilde x:=x-\bar x\in \widetilde H^{\frac{1}{2}}(\R/\Z,\R^N)\,,$$ so that $z(t)=(\bar x+\tilde x(t),y(t))$. Then, that the search of geometrically distinct $1$-periodic solutions of $(HS)$ is equivalent to finding critical points of the action functional 
$$\Phi:\mathcal M\to\mathbb R\,,\qquad\Phi((\tilde x,y),\bar x):=\frac{1}{2}\int_0^1({\dot z}(t),|J z(t))\,dt+\int_0^1 H(t,z(t))dt\,.$$
 
 \medbreak
 
 It can be shown \cite[Section 6]{Rab} that there exists a bounded selfadjoint linear operator $L:E\to E$ such that $\int_0^1({\dot z}(t)|J z(t))\,dt=\langle L(\tilde x,y),(\tilde x,y)\rangle_E$. Moreover, the Hilbert space $E$ splits in the form (\ref{splitting}) for some infinite-dimensional closed subspaces $E_\pm$ satisfying (\ref{aboutL}) and $E_0:=\{0\}\times\mathbb R^N$. Thus, $\Phi=\Phi_\psi$, where $$\psi:\mathcal M\to\mathbb R\,,\qquad \psi((\tilde x,y),\bar x):=\int_0^1 H(t,\bar x+\tilde x(t),y(t))dt$$
is the Nemytski\u{\i} functional associated to $H$.

\medbreak

Since $H^{\frac{1}{2}}(\R/\Z,\R^N)\subseteq L^2(\R/\Z,\R^N)$ and the inclusion is (completely) continuous, one promptly checks that if the admissible Hamiltonian $H$ is bounded then $\Phi$ is continuous, if $\nabla H$ is bounded then $\psi$ has class $C^1$, and if $H$ is twice continuously differentiable with respect to $z=(x,y)$ and its associated Hessian map ${\rm Hess}_z\psi$ is bounded, then $\psi$ has class $C^2$.   In addition, assumptions {\bf [}${\bm H_{1-3}}${\bf ]}-{\bf [}${\bm H_{2}^*}${\bf ]} imply that $\psi$ belongs to the class $\mathscr A$, while assumptions {\bf [}${\bm H_{1-6}}${\bf ]}-{\bf [}${\bm H_{2}^*}${\bf ]} imply that $\psi$ belongs to the class $\mathscr A^+$. It completes the proof of Theorem~\ref{szumod}. 
\end{proof}

\end{section}

\begin{section}{Relative category in infinite dimensions}\label{sec6}

This section, which follows along the lines of \cite{Szu1}, is aimed to prove the first part of Theorem \ref{absth1}. Accordingly, we assume that the Hilbert space $E=E_-\oplus E_0\oplus E_+$, the  finite-dimensional manifold $\mathcal V$, and the linear operator $L:E\to E$ are as required in Section \ref{sec5}. In addition, fix some  functional $\psi:\mathcal M:=E\times\mathcal V\to\mathbb R$ in the class $\mathscr A$, and write, for simplicity, $\Phi:=\Phi_\psi$.

\medbreak

 Let $\varXi=\{z\in \MM:\nabla_{\MM}\Phi(z)=0\}$ be the set of critical points of $\Phi$. By a pseudogradient vector field $V$ associated to $\Phi$ we mean  a locally Lipschitz-continuous map $V=(V_1,V_2):\mathcal M\backslash\varXi\to E\times T{\mathcal V}$ satisfying
  $$
 V_2(z)\in T_v{\mathcal V},\qquad \|V(z)\|\le 2\|\nabla_{\MM}\Phi(z)\|,\qquad   \langle\nabla_{\MM}\Phi(z),V(z)\rangle\ge\half\|\nabla_{\MM}\Phi(z)\|^2\,,
 $$
for every $z=(p,v)\in\mathcal M\backslash\varXi$. It can be seen~\cite[pp.~730--731]{Szu1} that a pseudogradient vector field $V=(V_1,V_2)$ can be found with first component of the form 
 \begin{equation}\label{V1}
 V_1(e,v)=Le+W(e,v),
 \end{equation}
 where $W:\MM\setminus\varXi\to E$  completely continuous and bounded. 
 
 \medbreak
 
 By a {\em deformation} of $\mathcal M$ we mean a continuous map $\eta:[0,1]\times\mathcal M\to\mathcal M$ with $\eta(0,m)=m$ for every $m\in\mathcal M$. The set of all deformations of $\mathcal M$ is, in some sense, too big to be useful, and for this reason we introduce a class of deformations below.
 \begin{definition}[Class $\mathcal D$]A deformation $\eta$ of $\mathcal M$ belongs to the class $\mathcal D$ provided that
  \begin{enumerate}
 	\item[$(\odot)$] for every $m\in \MM\setminus\varXi$, the curve $\eta(\cdot\,,m)$ moves forward along the integral curves of $V$, in the sense that for any $t_1<t_2$ in $[0,1]$, there are $\tau_1\le\tau_2$ in the domain $[0,\omega[$ of the integral curve $\gamma$ with $\gamma(0)=m$, such that $\eta(t_1,m)=\gamma(\tau_1)$ and $\eta(t_2,m)=\gamma(\tau_2)$;
 	\item[$(\odot\odot)$] Setting 	$\eta(t,e,v)=\big(\eta_1(t,e,v)\,,\eta_2(t,e,v)\big)$, one has
 	\begin{equation}\label{eta1}
 	\eta_1(t,e,v)=q_-(t,e,v)\pi_-(e)+q_+(t,e,v)\pi_+(e)+K(t,e,v)\,,\quad (t,e,v)\in[0,1]\times E\times{\cal V}\,,
 	\end{equation}
 	for some continuous functions $q_\pm:[0,1]\times \MM\to\R$ which are bounded on bounded sets and satisfy  $q_\pm(0,e,v)=1$, and a completely continuous map $K:[0,1]\times\MM\to E$ with $K(0,e,v)=\pi_0(e)$.
 	
 \end{enumerate}
\end{definition}
 \medbreak
 
 Our class $\mathcal D$ is similar to the homonymous one considered by Szulkin in~\cite{Szu1}; however, we have slightly relaxed some of his requirements for simplicity. Thus, our class contains more deformations than Szulkin's;  however, we will keep the essential properties of his relative category.
 
 \medbreak
 
Having delimited the family of allowed deformations, we are now ready to construct an associated notion of relative category. The definition below is due to Szulkin \cite{Szu1}.
\begin{definition}
	[Relative category with respect to the class $\mathcal D$]Given two closed sets $A,N\subseteq\MM$ we shall say that $A$ is of category $k\ge 0$ relative to ${N}$ and $\mathcal D$, written ${\rm cat}_{\MM,{N}}^{\cal D}(A)=k$, provided that $k$ is the smallest integer such that
	\begin{equation}\label{cates00}
	\text{$A=A_0\cup A_1\cup\ldots\cup A_k$,\qquad\ \  $A_0,\ldots,A_k$ closed and $A_1,\ldots, A_k$ contractible in $M$,}
	\end{equation}
	and there exists a deformation $\eta_0\in\mathcal D$ satisfying 
	\begin{equation}\label{eta00}
	\eta_0(\{1\}\times A_0)\subseteq N\,,\;\text{ and }\;\; \eta_0([0,1]\times N)\subseteq N\,.
	\end{equation} 
	If no such a $k$ exists, ${\rm cat}_{\MM,{N}}^{\cal D}(A):=+\infty$.
	
\end{definition}

Let us collect some selected properties of the relative category with respect to $\mathcal D$. We denote by
$
\Phi^c=\{z\in \MM:\Phi(z)\le c\}$ to the sublevel sets of $\Phi$. The closed set $N_1\subseteq M$ is called invariant for the class $\mathcal D$ if $\eta([0,1]\times N_1)\subseteq N_1$ for any $\eta\in\mathcal D$.
 
 \begin{lemma}\label{leem1}
 {The following hold:
 	\begin{enumerate}
 		\item[(i)] $\cat^{\mathcal D}_{M,N}(A_2)\leq\cat^{\mathcal D}_{M,N}(A_1)$, for any closed sets $A_2\subseteq A_1\subseteq M$.
 		\item[(ii)] Let $N_2\subseteq N_1\subseteq M$ be closed and assume that $N_1$ is invariant for the class $\mathcal D$. Then, $\cat^{\mathcal D}_{M,N_1}(A)\leq\cat^{\mathcal D}_{M,N_2}(A)$, for any closed set $A\subseteq M$.
\item[(iii)] Let the levels $a<b$ be such that the Palais\,--\,Smale condition (PS)$_c$ holds for every $c\in[a,b]$.  Then, $\Phi$ has at least ${\rm cat}_{\MM,\Phi^a}^{\cal D}(\Phi^b)$ critical points in $\Phi^{-1}([a,b])$.
\end{enumerate}}

\begin{proof}
{\em (i)} and {\em (ii)} are immediate from the definitions.   To check
{\em (iii)} it suffices to transcribe the proof of~\cite[Proposition 3.2]{Szu1}, in which the Palais\,--\,Smale condition is assumed for all levels and not only for the candidate critical ones.
\end{proof}

\end{lemma}

A key step in our argument towards the proof of Theorem~\ref{absth1}{\bf\em (a)} will be a refinement of~\cite[Proposition 3.6]{Szu1} which we examine next. We set
$$A(R)=\overline{\B_R^{E_-}}\times{\cal V}\,,\quad N(R):=\big((E_-\setminus \B_{R}^{E_-})+ E_0+ E_+\big)\times{\cal V}\,.$$
\begin{lemma}\label{lema}
	{ ${\rm cat}_{\MM,N(R)}^{\cal D}\big(A(R)\big)\ge {\rm cl}({\cal V})+1\,,\text{ for any }R>0.$}
\begin{proof}
	Using a contradiction argument, assume the existence of some $R>0$ such that  $${\rm cat}_{\MM,N(R)}^{\cal D}\big(A(R)\big)\le k:={\rm cl}({\cal V})\,.$$ It means that there are closed sets $A_0,A_1,\ldots,A_k\subseteq M$ and a deformation $\eta_0\in\mathcal D$ satisfying  (\ref{cates00})-(\ref{eta00}) for $A=A(R)$ and $N=N(R)$.
	We write $\eta_0=(\eta_1,\eta_2)$, with
	$\eta_1$ having the form~(\ref{eta1})
	for some continuous functions $q_\pm:[0,1]\times \MM\to\R$  with $q_\pm(0,e,v)=1$, and a completely continuous map $K:[0,1]\times\MM\to E$ with $K(0,e,v)=\pi_0(e)$, for every $(e,v)\in \MM$. 
	
	\medbreak
	
	The map $K$ being completely continuous, there are finite-dimensional subspaces $F_\pm\subseteq E_\pm$,  and a continuous map $C:[0,1]\times \MM\to F:=F_-\oplus E_0\oplus F_+$, with
	\begin{equation}\label{K2}
	\|K(t,e,v)-C(t,e,v)\|<\half\,,\;\hbox{ for every }(t,e,v)\in[0,1]\times A(R)\,.
	\end{equation}
	Furthermore, since $K(0,e,v)=\pi_0(e)$, after possibly replacing $C(t,e,v)$ by $C(t,e,v)-C(0,e,v)+\pi_0(e)$, we see that $C$ can be taken with $C(0,e,v)=\pi_0(e)$,  for every $(e,v)\in \MM$. We consider the sets
	$$
	\MM^*:=F_-\times{\cal V}\,,\quad N_2:=\mathbb S^{F_-}\times{\cal V}\,,\quad N_1:=\left(\overline{\mathbb B^{F_-}_{3/2}}\backslash\mathbb B^{F_-}_{1/2}\right)\times{\cal V}\,,\quad
	A^*:=A(R)\cap \MM^*=\overline{\B^{F_-}}\times\mathcal V\,,
	$$
	and $A_j^*=A_j\cap \MM^*$, for $j=0,1,\dots,\kappa$. Observe that $A^*=A_0^*\cup A_1^*\cup\dots\cup A_\kappa^*$, and $A^*_j$ is either empty or contractible in $\MM^*$, for every $j=1,\dots,\kappa$. We define $\eta^*:[0,1]\times \MM^*\to \MM^*$ by the rule
	$$
	\eta^*(t,f_-,v)=\Big(q_-(t,f_-,v)f_-+\pi_-\left( C(t,f_-,v)\right)\,,\,\eta_2(t,f_-,v)\Big)\,.
	$$
	Notice that $N_2$ is a strong deformation retract of $N_1$ relative to $M$ and $\eta^*$ is a deformation of $\mathcal M^*$ satisfying
	$$
	\eta^*([0,1]\times N_2)\subseteq N_1\,,\;\text{ and }\;\; \eta^*(\{1\}\times A_0^*)\subseteq N_1\,.
	$$
Thus,  Lemma \ref{leem2}{\em (ii)} in the Appendix 	implies that ${\rm cat}_{\MM,N_2}\big(A^*\big)\le k=\cl(\mathcal V)$. This contradicts  Lemma \ref{leem2}{\em (iii)} and completes the proof.\end{proof}

\end{lemma}

{\em Proof of Theorem~\ref{absth1}}
{{\bf\em (a).}} Since $\nabla\psi$ and the vector field $W$ in (\ref{V1}) are bounded, there exists some $R>0$ such that $N(R)$ is invariant for the class $\mathcal D$. Using now the boundedness of $\Phi$ one can find some level $a<\ell$ such that  $\Phi^a\subseteq N(R)$.
An easy argument shows that  $\displaystyle \sup_{E_-\times{\cal V}}\Phi<\ell$, allowing us to choose a second level $b$ with $\max\{a,\sup_{E_-\times{\cal V}}\Phi\}<b<\ell$\,. Then, $A(R)\subseteq\Phi^b$, and, by Lemma~\ref{leem1}{\em (i)}\,,
$$
{\rm cat}_{\MM,\Phi^a}^{\cal D}(\Phi^b)\ge {\rm cat}_{\MM,\Phi^a}^{\cal D}(A(R))\,,
$$
while, $N(R)$ being invariant for the class $\mathcal D$, Lemma ~\ref{leem1}{\em (ii)} gives
$$
{\rm cat}_{\MM,\Phi^a}^{\cal D}(A(R))\ge{\rm cat}_{\MM,N(R)}^{\cal D}(A(R))\,,
$$
and combining these inequalities with Lemma~\ref{lema}, we obtain
$$
{\rm cat}_{\MM,\Phi^a}^{\cal D}(\Phi^b)\ge {\rm cl}({\cal V})+1\,.
$$
Since the Palais-Smale condition holds at all levels $c\in[a,b]$ (by Lemma ~\ref{lemmaPS}), Lemma~\ref{leem1}{\em (iii)} provides the existence of at least $ {\rm cl}({\cal V})+1$ critical points of $\Phi$ and concludes the proof. \qed

\end{section}

\begin{section}{Morse theory and multiplicity of nondegenerate critical points}\label{sec7}

The aim of the next two sections is to prove  Theorem~\ref{absth1}{\bf\em (b)}: if $\psi$ belongs to the class $\mathscr A^+$, then it has at least ${\rm sb}(\mathcal V)$ critical points. Our approach will be divided into two steps. In this section we shall define a new class $\mathscr A^*$ of functionals $\psi$, which will be contained in $\mathscr A^+$, and we shall prove a version of Theorem~\ref{absth1}{\bf\em(b)} for functionals $\psi$ in this subclass. The second step, in Section \ref{sec8}, will consist in showing that, given some functional $\psi$ in the class $\mathscr A^+$, then there exists another functional $\psi^*$ in the class $\mathscr A^*$ such that the associated functionals $\Phi_\psi$ and $\Phi_{\psi^*}$ have exactly the same number of critical points. This will be shown provided only that $\Phi_\psi$ has finitely many critical points, and after that, the proof of Theorem~\ref{absth1} will be complete. 

\medbreak

In order to keep the notation within reasonable limits, we define, for every $R>0$,
$$B^-_R:=\Big(\mathbb B_{R}^{E_-}+E_{0}+E_+\Big)\times\mathcal V\,,\qquad B^+_R=\Big({E_-}+E_{0}+\mathbb B_{R}^{E_+}\Big)\times\mathcal V\,.$$
Correspondingly, we denote $\overline{B^-_R}:=\Big(\overline{\mathbb B_{R}^{E_-}}+E_{0}+E_+\Big)\times\mathcal V$, $\overline{B^+_R}=\Big({E_-}+E_{0}+\overline{\mathbb B_{R}^{E_+}}\Big)\times\mathcal V$, and $\partial {B^-_R}:=\Big(\mathbb S_{R}^{E_-}+E_{0}+E_+\Big)\times\mathcal V$.

\medbreak

 Pick now some functional $\psi:\MM\to\R$ in class $\mathscr A^+$; by combining (\ref{L}) and the boundedness of  $\nabla_E\psi$ we may choose $R_1>0$ so that
\begin{equation}\label{eu314}
\begin{cases}
\vspace{2mm}
(e,v)\in\mathcal M\backslash B^-_{R_1}\quad\Rightarrow\quad\langle Le+\nabla_E\psi(e,v),\pi_-(e)\rangle<0\,,\\ 
(e,v)\in\mathcal M\backslash B^+_{R_1}\quad\Rightarrow\quad\langle Le+ \nabla_E\psi(e,v),\pi_+(e)\rangle>0\,.
\end{cases}
\end{equation} 

\begin{definition}[Class $\mathscr A^*$]\label{def*}
	We shall say that $\psi$  belongs to the class $\mathscr A^*$ provided that the space $E$ is finite-dimensional and there exists a compact set $K\subseteq E_0$ satisfying:
\begin{enumerate}
 \item[{\em (i)}] all critical points of $\Phi_\psi$ belong to $C:=(E_-+K+E_+)\times\mathcal V$; \item[{\em (ii)}] $b:=\sup_{C\cap B^+_{R_1}}\Phi_\psi<\ell$;
 \item[{\em (iii)}] $K$ is a strong deformation retract of $E_0$; more precisely, there exists a deformation $m$ of $E_0$ with
  $$m(t,k)=k\text{ if }k\in K\,,\quad m(\{1\}\times E_0)= K\,,\quad \psi(\hat m\big(t,(e,v))\big)\leq\psi(e,v)\text{ if } (e,v)\in\mathcal M\,,$$
where $\hat m$ is the deformation of $\mathcal M$ given by
\begin{equation}
\label{hh2}
\hat m(t,(e,v)):=\big(\pi_-(e)+m(t,\pi_0(e))+\pi_+(e),v\big)\,,\qquad (t,(e,v))\in[0,1]\times\mathcal M\,.
\end{equation}

\end{enumerate}
\end{definition}

\begin{proposition}\label{key22}
{If $\psi$ belongs to the class $\mathscr A^*$, then  $\Phi_\psi$ has at least ${\rm sb}(\mathcal V)$ critical points.}
\begin{proof} Remembering that $\psi$ is bounded and writing for simplicity $\Phi:=\Phi_\psi$, we fix $R_2>R_1$ and $a\in\mathbb R$ so that
$\sup_{{B^+_{R_1}}\backslash B^-_{R_2}}\Phi<a<\inf_{B^-_{R_1}}\Phi$. Since $B_{R_1}^-\cap C\cap B_{R_1}^+\not=\emptyset$ it follows that $a<b$. Combining the fact that $\Phi$ satisfies the Palais-Smale condition at all levels $c\in[a,b]$ (by Lemma~\ref{lemmaPS}) with  Lemma~\ref{lemmmf} in the Appendix, we only have to check that ${\rm sb}(\Phi^b,\Phi^a)\ge{\rm sb}({\cal V})$. 

\medbreak

{\em  Claim I:} $H_*(\Phi^b,\Phi^a)\cong H_*(\Phi^b\cap\overline{B^+_{R_1}} ,\Phi^a\cap\overline{B^+_{R_1}}\backslash B^-_{R_2})$. To see this we consider the deformation $\eta_1$ of $\Phi^b$ defined by
$$\eta_1(t,(e,v)):=\Big((1-t)e+t r_1(e),v\Big)\,,$$
where $r_1(e)=r_-(\pi_-(e))+\pi_0(e)+r_+(\pi_+(e))$ and  $r_\pm:E_\pm\to E_\pm$ are defined by
$$r_+(e_+):=\begin{cases}
e_+&\text{if }\|e_+\|\leq R_1\,,\\
R_1e_+/\|e_+\|&\text{otherwise}\,,
\end{cases}\qquad\qquad  r_-(e_-):=\begin{cases}
(R_2/R_1)e_-&\text{if }\|e_-\|\leq R_1\,,\\
R_2\,e_-/\|e_-\|&\text{if }R_1<\|e_-\|<R_2\,,\\
e_-&\text{if }\|e_-\|\geq R_2\,.
\end{cases}$$
Now, Claim I follows from Lemma \ref{homotopy} in the Appendix (set $M:=\Phi^b$, $A:=\Phi^a$, $M':=\Phi^b\cap\overline{B^+_{R_1}}$, and  $A':=\Phi^a\cap\overline{B^+_{R_1}}\backslash B^-_{R_2}$).

\medbreak

{\em  Claim II:} $H_*(\Phi^b\cap\overline{B^+_{R_1}} ,\Phi^a\cap\overline{B^+_{R_1}}\backslash B^-_{R_2})\cong H_*(\Phi^b\cap\overline{B^+_{R_1}}\cap\overline{B^-_{R_2}} ,\Phi^a\cap\overline{B^+_{R_1}}\cap \partial B^-_{R_2})$. To check this statement we consider the deformation $\eta_2$ of $\Phi^b\cap\overline{B^+_{R_1}}$ defined by
$$\eta_2(t,(e,v)):=((1-t)e+tr_2(e),v)\,,$$
where $r_2(e)=\rho(\pi_-(e))+\pi_0(e)+\pi_+(e)$ and
	$$\rho(e_-):=\begin{cases}e_-&\text{if }\|e_-\|\leq R_2\,,\\
	R_2\, e_-/\|e_-\|&\text{if }\|e_-\|\leq R_2\,.
	\end{cases}
	$$
	The claim follows again from Lemma \ref{homotopy} with $M:=\Phi^b\cap\overline{B_{R_1}^+}$, $M':=\Phi^b\cap\overline{B_{R_1}^+}\cap\overline{B_{R_2}^-}$, $A:=\Phi^a\cap\overline{B^+_{R_1}}\backslash B^-_{R_2}$ and  $A'=\Phi^a\cap\overline{B^+_{R_1}}\cap \partial B^-_{R_2}$.

	\medbreak
	
	{\em  Claim III:} $H_*(\Phi^b\cap\overline{B^+_{R_1}}\cap\overline{B^-_{R_2}} ,\Phi^a\cap\overline{B^+_{R_1}}\cap \partial B^-_{R_2})\cong H_*(\mathbb B^{E_-}_{R_2}\times K\times\mathbb B^{E_+}_{R_1},\mathbb S^{E_-}_{R_2}\times K\times\mathbb B^{E_+}_{R_1})$. To check this claim we observe that the homotopy $\hat m$ (defined as in (\ref{hh2})), can be seen as a deformation of $\Phi^b\cap\overline{B^+_{R_1}}\cap\overline{B^-_{R_2}}$. The result follows from Lemma \ref{homotopy} with $M=\Phi^b\cap\overline{B^+_{R_1}}\cap\overline{B^-_{R_2}}$, $A:=\Phi^a\cap\overline{B^+_{R_1}}\cap \partial B^-_{R_2}$, $M'=\mathbb B^{E_-}_{R_2}+ K_2+\mathbb B^{E_+}_{R_1}$, and $A'=\mathbb S^{E_-}_{R_2}+ K_2+\mathbb B^{E_+}_{R_1}$. 
	\medbreak
	
		{\em  The end of the proof.} By the K\" unneth formula (Lemma~\ref{kun}), for $n=0,1,2,\dots$ we have
	\begin{equation*}
	H_n(\mathbb B^{E_-}_{R_2}\times K\times{\cal V},\mathbb S^{E_-}_{R_2}\times K\times{\cal V})\cong\prod_{i=0}^n
	[H_i(\mathbb B^{E_-}_{R_2}\times\mathcal V,\mathbb S^{E_-}_{R_2}\times\mathcal V)\otimes H_{n-i}({K})]\,,
	\end{equation*}
	and hence, the corresponding Betti numbers satisfy
	$$
	\dim H_n((\mathbb B^{E_-}_{R_2}+ K)\times{\cal V}\,,(\mathbb S^{E_-}_{R_2}+ K)\times{\cal V})\ge \dim H_n(\mathbb B^{E_-}_{R_2}\times{\cal V},\mathbb S^{E_-}_{R_2}\times{\cal V}).
	$$ 
	Using again the 
	K\" unneth formula (and combining it with Lemma~\ref{sph76}), we get
	\beq
	H_n(\mathbb B^{E_-}_{R_2}\times{\cal V},\mathbb S^{E_-}_{R_2}\times{\cal V})&&\cong\prod_{i=0}^n[H_i(\mathbb B^{E_-}_{R_2},\mathbb S^{E_-}_{R_2})\otimes H_{n-i}({\cal V})]\\
	&&\cong \R\otimes H_{n-r}({\cal V})\cong  H_{n-r}({\cal V})\,,
	\eeq
	where $r=\dim E_-$. We conclude that
	$\dim 
	H_n(\Phi_\psi^b,\Phi_\psi^a)\ge\dim H_{n-r}({\cal V})$, for any $n\geq r$, and the result follows.
\end{proof}

\end{proposition}

\section{From the class $\mathscr A^+$ to the class $\mathscr A^*$}\label{sec8}

Going back to the general framework established in Section \ref{sec5}, we fix some (possibly infinite-dimensional) Hilbert space $E=E_-\oplus E_0\oplus E_+$, some compact finite-dimensional manifold $\mathcal V$, and some functional $\psi:\mathcal M=E\times\mathcal V\to\mathbb R$ in the class $\mathcal A^+$. Through this section we shall further assume that 
$\Phi_{{\psi}}$ has finitely many critical points, as otherwise there is nothing to prove. To simplify the notation we shall write $F:=F_-\oplus E_0\oplus F_+$ and $\mathcal M^*:=F\times\mathcal V$. 
The goal of this section is to show the following:
\begin{proposition}\label{prop7}
	{There are finite-dimensional subspaces $F_\pm\subseteq E_\pm$ and a functional $\psi^*:\mathcal M^*\to\mathbb R$ in the class $\mathcal A^*$ such that  $\Phi_{\psi}:\mathcal M\to\mathbb R$ and $\Phi_{\psi^*}:\mathcal M^*\to\mathbb R$ have the same number of critical points.}
\end{proposition}

{\em First Step: Approximation by a finite-dimensional problem.}

\medbreak

We start by showing that it suffices to prove Theorem~\ref{absth1}{\bf\em (b)} when the space $E$ is finite-dimensional. With this goal we use a finite-dimensional reduction procedure which goes back to Conley and Zehnder \cite{ConZeh}. 
\begin{lemma}\label{lem8}{There are finite-dimensional subspaces $F_\pm\subseteq E_\pm$ and a $C^1$ map $G:\mathcal M^*\to F^\perp$ such that, letting  
\begin{equation}
\label{ups}\Upsilon:\mathcal M^*\to\mathcal M\,,\  (f,v)\mapsto(f+G(f,v),v)\,,\qquad\ \widehat{\psi}:=\psi\circ\Upsilon+\frac{1}{2}\langle LG,G\rangle:\mathcal M^*\to\mathbb R\,,
\end{equation}
then $\widehat{\psi}\in C^2(\mathcal M^*)$ belongs again to the class $\mathscr A^+$ and $\Upsilon$ establishes a $1:1$ correspondence between  the critical points of $\Phi_{\widehat{\psi}}$ and the critical points of $\Phi_{{\psi}}$.
}
\begin{proof}
In view of assumption {\bf [}${\bm \psi_4}${\bf ]}, the map  $\nabla_E\psi:\MM\to E$ is Lipschitz continuous in its first variable;  we denote by  $\alpha>0$ to an associated Lipschitz constant. By assumption {\bf [}${\bm \psi_2}${\bf ]}, the set $\mathfrak K:=\{\nabla_E\psi(e_1,v)-\nabla_E\psi(e_2,v):e_1,e_2\in E,\ v\in\mathcal V\}\subseteq E$ is relatively compact; therefore, there exist finite-dimensional subspaces $F_\pm\subseteq E_\pm$ such that $\max_{k\in\mathfrak K}\dist(k,F)<\frac{1}{2\alpha}$ (recall that $F:=F_-\oplus F_0\oplus F_+$). In addition, $L^{-1}=L:F^\perp\to F^\perp$ is non-expansive, by (\ref{L}). It follows that, letting $\pi:E\to F$ be the orthogonal projection, the map $$L^{-1}\circ(\rm{Id}_E-\pi)\circ\psi:E\times\mathcal V\to F^\perp$$ is contractive in the first variable, with associated Lipschitz constant $1/2$.

\medbreak

We use the splitting $E=F\oplus F^\perp$ to rewrite the points of $E$ as pairs  $(f,g)$, where $f\in F$ and $g\in F^\perp$. In this way, the critical points of $\Phi_\psi$ correspond to the solutions $(f,g,v)$ of the system
$$
\begin{cases}
Lf+\pi[\nabla_E\psi(f+g,v)]=0\,,\\
g+L^{-1}({\rm Id}-\pi)[\nabla_E\psi(f+g,v)]=0\,,\\
\nabla_{\cal V}\psi(f+g,v)=0\,.
\end{cases}
$$

 The Banach Contraction Theorem ensures that for any given $(f,v)\in F\times\mathcal V$, the middle equation of this system has a unique solution  $g=G(f,v)$ on $F^\perp$. Moreover, $G:F\times\mathcal V\to F^\bot$ is $C^1$-smooth (by the Implicit Function Theorem), and satisfies
 \begin{equation*}
 \lim_{\underset{e_0\in E_0}{\|e_0\|\to\infty}}G(e_0+b,v)=0\,,\qquad \lim_{\underset{e_0\in E_0}{\|e_0\|\to\infty}} G'_F(e_0+b,v)=0\,.
 \end{equation*}
Let $\Upsilon$ and $\widehat{\psi}$ be defined as in (\ref{ups}) and set $\Phi_{\widehat{\psi}}(f,v):=\half\big\langle Lf,f\big\rangle +\widehat\psi(f,v),\  (f,v)\in\mathcal M^*$. Then, $\Phi_{\widehat{\psi}}=\Phi_{{\psi}}\circ\Upsilon$. Moreover,
$$
\nabla_F\widehat\psi=\pi\circ(\nabla_E\psi)\circ\Upsilon\,,\qquad \nabla_{\mathcal V}\widehat\psi=(\nabla_{\mathcal V}\psi)\circ\Upsilon\,,
$$
so that $\widehat\psi$ is actually $C^2$-smooth and the following hold:
$$
\nabla_{\MM^*}\Phi_{\widehat\psi}=(\nabla_\MM\,\Phi_\psi)\circ\Upsilon,\qquad  \hess_{F}\Phi_{\widehat\psi}(f,v)=\hess_{E}\Phi_\psi\big(\Upsilon(f,v)\big)\circ\Upsilon'(f,v)\,.$$ It is now easy to check that $\widehat\psi$ satisfies assumptions {\bf [}${\bm \psi_{1-5}}${\bf ]} and the critical points of $\Phi_{\widehat\psi}$ are in a one-to-one correspondence (given by $\Upsilon$) with the critical points of $\Phi_\psi$. It completes the proof of Lemma \ref{lem8} and concludes  the discussion of the first step.
\end{proof}
\end{lemma}
\bigbreak

{\em Second Step: Intensifying the saddle geometry.}

\medbreak

In view of Lemma \ref{lem8} from now on we shall assume, without loss of generality, that $E=F$ is finite-dimensional. We write $\tilde E:=E_-\oplus E_+$ and use the splitting $E=E_0\oplus\tilde E$ to rewrite the points of $E$ in the form $e=e_0+\tilde e$, where $e_0=\pi_0(e)\in E_0$ and $\tilde e=\pi_-(e)+\pi_+(e)\in\tilde E$.
\begin{lemma}\label{lrt}
	{There exists some $R>0$ such that
$\nabla_{\widetilde E}\Phi_\psi(e_0+\tilde e,v)\not=0$ if $\|e_0\|\ge R\text{ and }\tilde e\not=0\,.$	
}\begin{proof}
Choose $R>0$ as in assumption {\bf [}${\bm \psi_{3}}${\bf ]}. After possibly replacing $R$ by a bigger number we may further assume that
\begin{equation}
\label{tnt0}\|\hess_E\psi(e_0+\tilde e,v)\|<\half\,,\;\text{ if }e_0\in E_0\backslash\mathbb B^{E_0}_R\text{ and }\tilde e\in\overline{\mathbb B^{\tilde E}_{2R_1}}\,,
\end{equation}
the constant $R_1>0$ being given as in (\ref{eu314}).
By (\ref{L}), $\|L\tilde e\|=\|\tilde e\|$ for any $\tilde e\in\tilde E$, and by {\bf [}${\bm \psi_{3}}${\bf ]}, $\nabla_{\widetilde E}\psi(e_0,v)=0$ whenever $e_0\in E_0\backslash\mathbb  B^{E_0}_R$. The triangle inequality gives
	\begin{equation}\label{lp}
\|\nabla_{\widetilde E}\Phi_\psi(e_0+\tilde e,v)\|\geq\|\tilde e\|-\|\nabla_{\widetilde E}\psi(e_0+\tilde e,v)-\nabla_{\widetilde E}\psi(e_0,v)\|\,.
	\end{equation}

	In view of (\ref{tnt0}), the map $\tilde e\mapsto\nabla_{\widetilde E}\psi(e_0+\tilde e,v)$ is contractive on $\overline{\B_{2R_1}^{\tilde E}}$, and~(\ref{lp}) implies that
	$
	\nabla_{\widetilde E}\Phi_\psi(e_0+\tilde e,v)\not=0$ provided that $\|e_0\|\ge R$ and $0<\|\tilde e\|\le 2R_1$. In combination with~(\ref{eu314}) this yields the result.
\end{proof}
\end{lemma}

\begin{lemma}\label{lem9}
{There exists a functional $\psi^*:\MM\to\R$ in the class ${\mathscr A^*}$ such that  $\Phi_\psi$ and $\Phi_{\psi^*}$ have the same critical points.}
\end{lemma}

\begin{proof}
Choose $R>0$ large enough so that it satisfies both Lemma \ref{lrt} and assumption {\bf [}${\bm \psi_{3}}${\bf ]}. Since we further assumed (at the beginning of this section) that the number of critical points of $\Phi_{{\psi}}$ is finite, after possibly replacing $R$ by some bigger constant, we will have
\begin{equation}\label{tnt}
0\not=\nabla_{\MM}\Phi_\psi(e_0,v)=\nabla_{\MM}\psi(e_0,v)\in E_0\text{ for any }e_0\in E_0\backslash\mathbb B_{R}^{E_0}\,.
\end{equation}

In particular, $\rho(e_0):=\psi(e_0,v)$ does not depend on $v\in\mathcal V$ if $\|e_0\|\geq R$. The function $\rho:E_0\setminus\B_R^{E_0}\to\R$ defined in this way is $C^2$-smooth and, by {\bf [}${\bm \psi_{1,2,4}}${\bf ]} and~(\ref{psi}),(\ref{tnt}),  it satisfies
\begin{equation}\label{tnt2}
\begin{cases}
\;\;\rho(e_0)<\ell,\quad \nabla\rho(e_0)\not=0\,,\ \ \text{ for any }e_0\in E_0\setminus\B^{E_0}_R\,,\\
\displaystyle\lim_{\|e_0\|\to\infty}\rho(e_0)=\ell\,,\; \lim_{\|e_0\|\to\infty}\nabla \rho(e_0)=0\,, \; \lim_{\|e_0\|\to\infty}\hess\, \rho(e_0)=0\,.
\end{cases}
\end{equation}

Choose now some number $\alpha\in]\max_{\mathbb S_R^{E_0}}\rho,\ell[$ and some $C^2$ function $\omega:\mathbb R\to\mathbb R$ with $\omega(r)=0$ if $r\leq\alpha$, $\omega'(r)>0$ if $r>\alpha$, and  $\omega(\ell)=1$. Then, the function $h:E_0\to\mathbb R$ defined by
$$h(e_0):=\begin{cases}
0&\text{if }\|e_0\|< R\\
\omega(\rho(e_0))&\text{if }\|e_0\|\geq R
\end{cases}\,,$$
is a basket function (in the sense of Definition \ref{bf}) for the compact hypersurface $\mathcal S:=h^{-1}(\alpha)$.
Pick now some number $R'>\max_{e_0\in\mathcal S}\|e_0\|$, and some $C^2$-smooth function $n:\R\to\R$ with 
$$
 n(\rho)=1\,,\;\text{ if }\rho\leq R'\,,\qquad n'(\rho)<0\,,\;\text{ if }R'<\rho<R'+1\,,\qquad
  n(\rho)=0\,,\;\text{ if }\rho\ge R'+1\,,
 $$
and define, for $\lambda>0$,
$$
\psi_\lambda:\mathcal M\to\mathbb R\,,\qquad (e_0+\tilde e,v)\mapsto n(\|e_0\|)\,\psi(e_0+\tilde e,v)+\lambda h(e_0) 
\,.
$$

We claim that, for $\lambda>0$ large enough, $\Phi_{{\psi}}$ and $\Phi_{{\lambda}}:=\Phi_{\psi_\lambda}$ have the same critical points. Indeed, 
\begin{itemize}
\item {\em $\Phi_{{\psi}}$ and $\Phi_{{\lambda}}$ coincide on $(\mathbb B_R^{E_0}+\tilde E)\times\mathcal V$.} This is clear from the definitions.
\item {\em No critical point of $\Phi_{{\psi}}$ belongs to $((E_0\backslash\mathbb B_R^{E_0})+\tilde E)\times\mathcal V$.} This statement follows from the combination of (\ref{tnt}) and Lemma \ref{lrt}.
\item {\em If $\lambda>0$ is large enough, no critical point of $\Phi_{{\lambda}}$ belongs to $((E_0\backslash\mathbb B_R^{E_0})+\tilde E)\times\mathcal V$.} In fact, one checks the inequality $\nabla_{E}\Phi_{\lambda}(e_0+\tilde e,v)\not=0$ by considering the following four possible cases: {\em (a):} if $R\leq\|e_0\|\leq R'$ and $\tilde e=0$, by using (\ref{tnt2}); {\em (b):} if $R\leq\|e_0\|\leq R'$ and $\tilde e\not=0$, by Lemma \ref{lrt};  {\em (c):} if $R'\leq\|e_0\|\leq R'+1$ and $\lambda>0$ is large enough, by the fact that $\nabla h$ is bounded away from zero in this set; {\em (d):} if $\|e_0\|\geq R'+1$, by (\ref{tnt2}).   

\end{itemize}

To conclude, we claim that, for $\lambda>0$ large enough, the functional $\psi_{\lambda}$ belongs to the class $\mathscr A^*$. This is now easy to check, both in what concerns to membership of the class $\mathscr A^+$ and requirements {\em (i)-(ii)} in Definition \ref{def*}. To check {\em (iii)} pick some number $\beta\in\,]\max_{\|e_0\|\leq R'+1}h,1[$ and consider the compact set  $K:=h^{-1}([0,\beta])=\overline{{\rm int}\,\mathcal S'}$, where  $\mathcal S':=h^{-1}(\beta)$.  Now, the deformation $m$ of $E_0$ can be built as follows:
we keep all the points of $K$ fixed, while, if $e_0\in E_0\setminus K$, then $m(\cdot,e_0)$ is the curve, starting from $m(0,e_0)=e_0$, which follows backwards the flow lines of $\nabla h$, and arrives at the point $m(1,e_0)$ where the flow first meets $\mathcal S'$. (Notice that this flow is transversal to $\mathcal S'$.) It completes the proof of Lemma \ref{lem9}.
\end{proof}

We end the argumentation by observing that the combination of Lemmas \ref{lem8} and \ref{lem9} implies Proposition \ref{prop7}. In view of Proposition \ref{key22}, part {\bf\em (b)} of Theorem \ref{absth1} follows.  
\end{section}

\begin{section}{Appendix: A review on relative category and relative Morse theory}\label{appendix}
	
	Relative category and relative Morse theory were developed to study existence and multiplicity of critical points of functionals which are unbounded from below.
	In this Appendix we collect some basic elements from these theories which were used in Sections \ref{sec6} and \ref{sec7}. We do not claim any originality in this section, which is  well-known to the specialists.
	
	\medbreak 
	\begin{subsection}{Relative category}\label{Ap1}
	It will be convenient to take a more general point of view and work on an arbitrary metric space $M$. Let $A,N\subseteq M$ be closed; by a {\em deformation of $A$ in $M$} we mean a continuous map $\eta:[0,1]\times A\to M$ with $\eta(0,a)=a$ for any $a\in A$. The set  $A$ is {\em contractible in $M$} provided that there exists a deformation $\eta$ of $A$ in $M$ such that $\eta(1,A)=\{p\}$ is a singleton. We shall say that $A$ is of category $k\ge 0$ relative to $N$, denoted ${\rm cat}_{M,N}(A)=k$, provided that $k$ is the smallest integer such that 
	\begin{equation}\label{cates}
	\text{$A=A_0\cup A_1\cup\ldots\cup A_k$,\qquad\ \  $A_0,\ldots,A_k$ closed and $A_1,\ldots, A_k$ contractible in $M$,}
	\end{equation}
	and there exists a deformation $\eta_0$ of $A\cup N$ in $M$ satisfying 
	\begin{equation*}
	\eta_0(\{1\}\times A_0)\subseteq N\,,\;\text{ and }\;\; \eta_0([0,1]\times N)\subseteq N\,.
	\end{equation*} 
	(If no such a $k$ exists, $\cat_{M,{N}}(A):=+\infty$). 
	
	\medbreak
	
	Let the closed sets $N_2\subseteq N_1\subseteq M$ be given. We say that $N_1$ is a {\em retract} of $M$ provided that there exists a continuous map $r:M\to N_1$ with $r(n)=n$ for $n\in N_1$. We say that $N_2$ is a {\em strong deformation retract of $N_1$ relative to $M$} provided that there exists a deformation $h$ of $M$ satisfying
	\begin{equation}\label{hh} h([0,1]\times N_1)\subseteq N_1\,,\qquad h(\{1\}\times N_1)\subseteq N_2\,,\;\text{ and }\;\;  h(t,n_2)=n_2\ \forall (t,n_2)\in[0,1]\times N_2\,.
	\end{equation} 

	For instance, the closed unit ball $\overline{\mathbb B^m}$ is a retract of $\mathbb R^m$, and the sphere $\mathbb S^{m-1}$ is a strong deformation retract of the shell $\overline{\mathbb B^m}\backslash\mathbb B_{1/2}^m$ relative to $\mathbb R^m$. However, $\mathbb S^{m-1}$ is not a retract of  $\overline{\mathbb B^m}$. 
	
	\begin{lemma}\label{leem2}
		{The following hold:
			\begin{enumerate}
				\item[(i)] If $N_1$ is a retract of $M$, then $\cat_{N_1,N_2}(N_1)\leq\cat_{M,N_2}(N_1)$.
				\item[(ii)] Assume (\ref{cates}), that $N_2$ is a strong deformation retract of $N_1$ relative to $M$, and the existence of a deformation $\eta$ of $A_0\cup N_2$ in $M$ with 
				\begin{equation}
				\label{eta}	\eta([0,1]\times N_2)\subseteq N_1\,,\;\text{ and }\;\; \eta(\{1\}\times A_0)\subseteq N_1.
				\end{equation} Then, $\cat_{M,N_2}(A)\leq k$.
				\item[(iii)] Let $\mathcal V$ be a finite-dimensional, compact, connected  manifold without boundary, and let $m\ge 1$ be an integer. Then, 
				$
				\cat_{{\mathbb R^m}\times\mathcal V,\mathbb S^{m-1}\times\mathcal V}(\overline{\mathbb B^m}\times\mathcal V)\ge\cl(\mathcal V)+1\,.
				$
				
		\end{enumerate}  }
		\begin{proof}
			Part {\em (i)} follows easily from the definitions. In order to check {\em (ii)} choose $h$ and $\eta$ as in (\ref{hh}) and (\ref{eta}), and denote by 
			$$d:M\to [0,1]\,,\qquad d(m):=\min(\dist(m,N_2),1)\,,$$
			to the distance function to $N_2$, with a maximum of $1$. Let $h$ be a deformation of $M$ with (\ref{hh}) and let $\alpha:[0,1]\times M\to M$ be the deformation of $M$ defined by
			$$\alpha(t,m):=\begin{cases}
			m, &\text{if }m\in N_2\,,\\
			h\left(\frac{t}{d(m)}\right),&\text{if }m\notin N_2\text{ and }0\leq t\leq d(m)\,,\\
			h(1,m),&\text{if }d(m)\leq t\leq 1\,.
			\end{cases}$$
			We use Tietze's extension theorem to find some continuous function $\delta:[0,1]\times M\to [0,1]$ satisfying
			$$\delta(t,m)=\begin{cases}
			0&\text{if }t=0\,,\\
			d(\eta(t,m))&\text{if }(t,m)\in\{1\}\times A_0\text{ or }m\in N_2\,,
			\end{cases}$$
			and define $\eta_*:[0,1]\times(A_0\cup N_2)\to M$ by
			$\eta_*(t,m):=\alpha(\delta(t,m),\eta(t,m))$. One easily checks that this is a deformation of $A_0\cup N_2$ into $M$ carrying $\{1\}\times A_0$ and $[0,1]\times N_2$ inside $N_2$, proving the result.
			
			\medbreak
			
			Finally, in order to obtain {\em (iii)} we observe that $\overline{\mathbb B^m}\times\mathcal V$ is a retract of ${\mathbb R^m}\times\mathcal V$, and combine {\em (i)} with ~\cite[Proposition 2.6 and Lemma 3.7]{Szu1} to find
			$$
			\cat_{{\mathbb R^m}\times\mathcal V,\mathbb S^{m-1}\times\mathcal V}(\overline{\mathbb B^m}\times\mathcal V)\ge\cat_{\overline{\mathbb B^m}\times\mathcal V,\mathbb S^{m-1}\times\mathcal V}(\overline{\mathbb B^m}\times\mathcal V)\ge\cl(\mathcal V)+1\,.
			$$
		\end{proof}
		\end{lemma}
\end{subsection}

\begin{subsection}{Relative homology and Morse theory}

Given a metric space $M$ and a subset $A\subseteq M$  we denote by  $H_*(M,A)=\{H_n(M,A)\}_{n\ge 0}$ to the associated sequence of relative homology groups with real coefficients.  Thus, $H_n(M,A)$ is, for each index $n$, a real vector space whose elements are equivalence classes of singular chains having zero boundary. Moreover, this is done in such a way that $H_*(M,\emptyset)=H_*(M)$. 

\medbreak

We shall write $f:(M,A)\to(N,B)$ meaning that $A\subseteq M$, $B\subseteq N$, and $f:M\to N$ is a continuous map with $f(A)\subseteq B$. Such a map induces a corresponding sequence of linear transformations $f_*:H_*(M,A)\to H_*(N,B)$, and this correspondence is functorial, in the sense that $\left[{\rm Id}_{(M,A)}\right]_*={\rm Id}_{H_*(M,A)}$ and, for any maps $f:(M,A)\to (N,B)$ and $g:(N,B)\to (P,C)$, one has that $(g\circ f)_*=g_*\circ f_*$. Furthermore, if two maps $f,g:(M,A)\to(N,B)$ are homotopic (i.e., there exists a continuous homotopy $h:[0,1]\times M\to N$ with $h(0,\cdot)=f$, $h(1,\cdot)=g$, and  $h([0,1]\times A)\subseteq B$), then $f_*=g_*$. See, e.g.,~\cite[Ch. 4]{Spa} for more details.

\medbreak

 A subset $A\subseteq M$ will be called {\em invariant} by the deformation $\eta$ of $M$ if $\eta([0,1]\times A)\subseteq A$. 
\begin{lemma}\label{homotopy}
	Let  the sets $M',A,A'\subseteq M$ with $A'\subseteq M'\cap A$ be invariant by some deformation $\eta$ of $M$ satisfying
	$$\eta(\{1\}\times M)\subseteq M',\qquad \eta(\{1\}\times A)\subseteq A'\,.$$
	Then, $H_*(M,A)\cong H_*(M',A')$.
\end{lemma}

\begin{proof}
	We denote by $i:(M',A')\to(M,A)$ to the inclusion, and consider the mapping $r:(M,A)\to(M',A')$ defined by $r(m):=\eta(1,m)$, for $m\in M$. Then,
	$$
	i_*\circ r_*=(i\circ r)_*={\rm Id}_{H_*(M,A)}\,,
	$$
	since $\eta:[0,1]\times M\to M$ is a continuous homotopy of $(M,A)$, connecting the identity with $i\circ r$. Moreover,
	$$
	r_*\circ i_*=(r\circ i)_*={\rm Id}_{H_*(M',A')}\,,
	$$
	since $\eta:[0,1]\times M'\to M'$ is also a continuous homotopy of $(M',A')$, connecting the identity with $r\circ i$. 
	Hence, $i_*:H_*(M',A')\to H_*(M,A)$ and $r_*:H_*(M',A')\to H_*(M,A)$ are mutually inverse isomorphisms.
\end{proof}

The so-called  K\" unneth formula relates the relative homology groups of a product of spaces with those of each factor, see, e.g.~\cite[p.~5]{Cha},~\cite[p.~235]{Spa}. We shall be interested only in the following particularly simple case ($\otimes$ denotes the usual tensor product):

\begin{lemma}[K\" unneth formula]\label{kun}{Let $A\subseteq M$ and $N$ be metric spaces. Then,
		\begin{equation*}
		H_n(M\times N,A\times N)\cong\prod_{i=0}^n
		[H_i(M,A)\otimes H_{n-i}({N})]\,,
		\end{equation*}
		for every integer $n\ge 0$.} 
\end{lemma}

The dimensions $\dim H_n(M,A)$ of the homology vector spaces are called {\em Betti numbers}. Under compactness and nondegeneracy conditions,  these numbers can be used to estimate the number of critical points lying between two level sets of a given functional.  More precisely, let $\phi$ be a $C^2$-smooth functional defined on some $C^2$-smooth, finite-dimensional manifold without boundary and having only nondegenerate critical points. Assume finally that $a<b$ are real numbers such that the Palais\,--\,Smale condition $(PS)_c$ holds at every level $c\in[a,b]$. The result below is a classical consequence of the so-called weak Morse inequalities, see, e.g.~\cite[Example 3 (p.~328) and Corollary 5.1.28 - Theorem~5.1.29 (p.~339)]{Cha1} or~\cite[Theorem 8.2 (p.~182) and Remark 2 (p.~189)]{Maw-Wil}. 

\begin{lemma}\label{lemmmf}
	{Under these conditions, the functional $\phi$ has at least $
		{\rm sb}(\phi_b,\phi_a)=\sum_{n=0}^\infty \dim H_n(\phi_b,\phi_a)
		$ critical points in $\phi^{-1}([a,b])$.
	}
\end{lemma}

We conclude this Appendix by listing the (well-known) relative homology groups of the pairs
$(\R^N,\mathbb S^{N-1})$ or, what is the same, (by Lemma~\ref{homotopy}), of the pairs $(\overline{\B^N},\mathbb S^{N-1})$. 

\begin{lemma}\label{sph76}
	{$H_n(\R^N,\mathbb S^{N-1})\cong H_n(\overline{\B^N},\mathbb S^{N-1})\cong\begin{cases}0\,,&\text{ if }n\not=N\,,\\
		\R\,,&\text{ if }n=N\,.
		\end{cases}
		$}
\end{lemma}
\end{subsection}
\end{section}

\vspace{0.75cm}

\ni Authors' addresses: 

\medbreak

\medbreak

Alessandro Fonda

Dipartimento di Matematica e Geoscienze

Universit\`a di Trieste

P.le Europa 1 

I-34127 Trieste 

Italy 

e-mail: a.fonda@units.it 

\medbreak

\medbreak

Antonio J. Ure\~na

Departamento de Matem\'atica Aplicada

Universidad de Granada

Campus de Fuentenueva

18071 Granada

Spain

e-mail: ajurena@ugr.es 

\medbreak

\medbreak

\ni Mathematics Subject Classification (2010): 34C25, 37J10

\medbreak

\ni Keywords: Poincar\'e\,--\,Birkhoff, periodic solutions, Hamiltonian systems.


\begin{thebibliography}{x}

\bibitem{Ama} H. Amann, A note on degree theory for gradient mappings.
Proc. Amer. Math. Soc. 85 (1982), 591--595. 

\bibitem{Arn} V. I. Arnol'd,  Mathematical Methods of Classical Mechanics. Second edition. Graduate Texts in Mathematics, 60, Springer, New York, 1989.
\bibitem{BanHur}

A. Banyaga and D. Hurtubise, Lectures on Morse Homology. Kluwer Texts in Math. Sci. 29, Dordrecht, 2004.

\bibitem{BenCapFor}  V. Benci, A. Capozzi and D. Fortunato, Periodic solutions of Hamiltonian systems with superquadratic potential. Ann. Mat. Pura Appl. (4) 143 (1986), 1--46.

\bibitem{Bir1} G. D. Birkhoff, Proof of Poincar\'e's geometric theorem. Trans. Amer. Math. Soc. 14 (1913), 14--22.

\bibitem{Bir3} G. D. Birkhoff, An extension of Poincar\'e's last geometric theorem. Acta Math. 47 (1925), 297--311.

\bibitem{Bir} G. D. Birkhoff, Une g\'en\'eralisation \`a $n$ dimensions du dernier th\'eor\`eme de g\'eom\'etrie de Poincar\'e. C. R. Acad. Sci., Paris 192 (1931), 196--198.

\bibitem{BosOrt} A. Boscaggin and R. Ortega, Monotone twist maps and periodic solutions of systems of Duffing type. Math. Proc. Cambridge Philos. Soc. 157 (2014), 279--296.

\bibitem{BroNeu} M. Brown and W. D. Neumann, Proof of the Poincar\'e\,--\,Birkhoff fixed point theorem. Michigan Math. J. 24 (1977), 21--31.

\bibitem{Cha0} K. C. Chang, On the periodic nonlinearity and the multiplicity of solutions. Nonlinear Anal. 13 (1989), 527--537. 

\bibitem{Cha} K. C. Chang,  Infinite-dimensional Morse Theory and Multiple Solution Problems. Progress in Nonlinear Differential Equations and their Applications, 6. Birkh\"{a}user, Boston, 1993.

\bibitem{Cha1} K. C. Chang, Methods in Nonlinear Analysis. Springer, New York, 2005.

\bibitem{ConZeh} C. C. Conley and E. J. Zehnder, The Birkhoff\,--\,Lewis fixed point theorem and a conjecture of V. I. Arnold. Invent. Math. 73 (1983), 33--49.

\bibitem{Die} J. Dieudonn\'e, \'El\'ements d'Analyse. Tome III: Chapitres XVI et XVII. Gauthier-Villars, Paris, 1970.

\bibitem{FonSabZan} A. Fonda, M. Sabatini and F. Zanolin, Periodic solutions of perturbed Hamiltonian systems in the plane by the use of the Poincar\'e\,--\,Birkhoff Theorem. Topol. Meth. Nonlinear Anal. 40 (2012), 29--52.

\bibitem{FonUre1} A. Fonda and A. J. Ure\~na,  A higher dimensional Poincar\'e-Birkhoff theorem for Hamiltonian flows. Ann. Inst. H. Poincar\'e Anal. Non Lin\'eaire 34 (2017), 679--698. 

 \bibitem{FouWil} G. Fournier and M. Willem, Relative category and the calculus of variations. Variational Methods (Paris, 1988), 95--104,
 Progr. Nonlin. Differential Equations Appl., 4. Birkh\"{a}user, Boston, 1990. 

\bibitem{Gol} C. Gol\'e, Symplectic Twist Maps. Global Variational Techniques. Advanced Series in Nonlinear Dynamics, 18. World Scientific, River Edge, 2001.

\bibitem{LeC} P. Le Calvez, About Poincar\'e\,--\,Birkhoff theorem. Publ. Mat. Urug. 13 (2011), 61--98.

 \bibitem{Liu} J. Q. Liu,  A generalized saddle point theorem. J. Differential Equations 82 (1989), 372--385. 

\bibitem{Maw-Wil} J. Mawhin and M. Willem, Critical Point Theory and Hamiltonian Systems. Applied Mathematical Sciences, 74, Springer, New York, 1989.

\bibitem{Mor}M. Morse, Differentiable mappings in the Schoenflies problem. Proc. Nat. Acad. Sci. U.S.A. 44 (1958), 1068--1072.

\bibitem{Mos} J. Moser, A fixed point theorem in symplectic geometry. Acta Math. 141 (1978), 17--34. 

\bibitem{MosZeh} J. Moser and E. J. Zehnder, Notes on Dynamical Systems. Amer. Math. Soc., Providence, 2005.

\bibitem{Poi} H. Poincar\'e, Sur un th\'eor\`eme de g\'eom\'etrie. Rend. Circ. Mat. Palermo 33 (1912), 375--407.

\bibitem{Rab} P. H. Rabinowitz, Minimax Methods in Critical Point Theory with Applications to Differential Equations. CBMS Regional Conference Series in Mathematics, 65, Amer. Math. Soc., Providence, 1986. 

\bibitem{Spa} E. H. Spanier, Algebraic Topology. McGraw-Hill, New York, 1966.

\bibitem{Sch} J. T. Schwartz, Nonlinear Functional Analysis.
Notes on Mathematics and its Applications, Gordon and Breach, New York, 1969.

\bibitem{Szu1} A. Szulkin, A relative category and applications to critical point theory for strongly indefinite functionals. Nonlinear Anal. 15 (1990), 725--739.

\bibitem{Szu2} A. Szulkin, Cohomology and Morse theory for strongly indefinite functionals. Math. Z. 209 (1992), 375--418.

\bibitem{Wei} A. Weinstein, Lectures on Symplectic Manifolds. CBMS Regional Conference Series in Mathematics, 29, Amer. Math. Soc., Providence, 1977.

\end{thebibliography}
\end{document}